\documentclass{interact}
\usepackage{amsmath}
\usepackage{amssymb}
\usepackage{amsthm}
\usepackage{bbm,bm}
\usepackage{graphicx}
\usepackage{xcolor}
\usepackage{subcaption}

\usepackage{enumerate}
\usepackage{url}
\usepackage{bm}
\usepackage{nicefrac}

\usepackage{algorithm}
\usepackage{algpseudocode}

\usepackage{soul} 
\usepackage{todonotes}

\usepackage{mathtools}
\mathtoolsset{showonlyrefs=true} 
\usepackage{hyperref}

\usepackage{lmodern}

\newcommand{\bmu}{\bm{u}}
\newcommand{\bmx}{\bm{x}}
\newcommand{\bmy}{\bm{y}}
\newcommand{\bmv}{\bm{v}}
\newcommand{\bmm}{\bm{m}}
\newcommand{\xstar}{\bmx_*}
\newcommand{\Tau}{T}
\newcommand{\Eta}{H}
\newcommand{\lrate}{\lambda}
\newcommand{\agsgd}{AGS-GD~}
\newcommand{\agssgd}{AGS-SGD~}
\newcommand{\agsadam}{AGS-Adam~}

\theoremstyle{definition}
\newtheorem{thm}{Theorem}[section]
\newtheorem{lem}[thm]{Lemma}
\newtheorem{cor}[thm]{Corollary}

\begin{document}

\title{Anisotropic Gaussian Smoothing for Gradient-based Optimization}
\author{
    \name{
        Andrew Starnes\textsuperscript{a}\thanks{CONTACT A. Starnes. Email: astarnes@lirio.com}
        and
        Guannan Zhang\textsuperscript{b}
        and
        Viktor Reshniak\textsuperscript{b}
        and
        Clayton Webster\textsuperscript{a}
    }
    \affil{
        \textsuperscript{a}Behavioral Research Learning Lab, Lirio, Inc.
        \textsuperscript{b}Computer Science and Mathematics Division, Oak Ridge National Laboratory;
    }
}

\maketitle

\begin{abstract}
This article introduces a novel family of optimization algorithms—Anisotropic Gaussian Smoothing Gradient Descent (AGS-GD), AGS-Stochastic Gradient Descent (AGS-SGD), and AGS-Adam—that employ anisotropic Gaussian smoothing to enhance traditional gradient-based methods, including GD, SGD, and Adam. The primary goal of these approaches is to address the challenge of optimization methods becoming trapped in suboptimal local minima by replacing the standard gradient with a non-local gradient derived from averaging function values using anisotropic Gaussian smoothing. Unlike isotropic Gaussian smoothing (IGS), AGS adapts the smoothing directionality based on the properties of the underlying function, aligning better with complex loss landscapes and improving convergence. The anisotropy is computed by adjusting the covariance matrix of the Gaussian distribution, allowing for directional smoothing tailored to the gradient's behavior. This technique mitigates the impact of minor fluctuations, enabling the algorithms to approach global minima more effectively. We provide detailed convergence analyses that extend the results from both the original (unsmoothed) methods and the IGS case to the more general anisotropic smoothing, applicable to both convex and non-convex, L-smooth functions. In the stochastic setting, these algorithms converge to a noisy ball, with its size determined by the smoothing parameters. The article also outlines the theoretical benefits of anisotropic smoothing and details its practical implementation using Monte Carlo estimation, aligning with established zero-order optimization techniques.

\end{abstract}



\section{Introduction}

A standard approach to solving optimization problems is to use gradient descent (GD) when the function is deterministic and stochastic gradient descent (SGD) or adaptive moment estimation (Adam) in the stochastic setting.
A common problem with these methods is that the path that they take can lead them into suboptimal basins that they cannot escape from.
The focus of this paper is to provide alternative optimization techniques that avoid this behavior.
In particular, we exchange the gradient of the function in GD, SGD, and Adam with a non-local gradient that comes from averaging out the values of the original function.
The resulting averaging function has the property that small fluctuations in the original function are smoothed out, which make gradient descent algorithms less likely to find non-global minima.
We refer to the anisotropic Gaussian smoothing algorithms modifying GD, SGD, and Adam as AGS-GD, AGS-SGD, and AGS-Adam, respectively.

Our goal is to minimize real-valued functions on $\mathbb{R}^d$, that is, we aim to solve the optimization problem
\begin{equation}
\label{eqn:min_of_f}
    \min_{\bmx\in\mathbb{R}^d}f(\bmx).
\end{equation}
We consider both the deterministic and stochastic settings.
In the stochastic setting, we are given $F:\mathbb{R}^d\times\Omega\to\mathbb{R}$ for some probability space $\Omega$ and a sample $\{\omega_1,...,\omega_K\}\subseteq\Omega$.
We denote
\begin{equation}
    f_{k}(\bmx)=F(\bmx,\omega_k):\mathbb{R}^d\to\mathbb{R}
\end{equation}
and assume $f:\mathbb{R}^d\to\mathbb{R}$ satisfies
\begin{equation}
    \mathbb{E}\big[f_k(\bmx)\big]=f(\bmx)
    \text{ and }
    \mathbb{E}\big[\nabla f_k(\bmx)\big]=\nabla f(\bmx).
\end{equation}
The stochastic variant of this problem is the focus of machine learning.

For any function $g:\mathbb{R}^d\to\mathbb{R}$ and symmetric, invertible matrix $\Sigma\in\mathbb{R}^{d\times d}$, we define the $\Sigma$-Gaussian smoothed version of $g$ as
\begin{align}
\label{eqn:definition_of_fsigma}
    g_{\Sigma}(\bmx)
    &=\frac{1}{\pi^{\frac{d}{2}}}\int_{\mathbb{R}^d}g\left(\bmx+\Sigma\bmu\right)e^{-\|\bmu\|^2}\;d\bmu.
\end{align}
There are two other convenient and equivalent formulations of $g_{\Sigma}$,
\begin{align}
\label{eqn:change_of_variables}
    g_{\Sigma}(\bmx)
    &=\frac{1}{\pi^{\frac{d}{2}}|\Sigma|}\int_{\mathbb{R}^d}g(\bmv)e^{-(\bmv-\bmx)^T\Sigma^{-2}(\bmv-\bmx)}\;d\bmv
    =(g\star k_{\Sigma})(\bmx)
\end{align}
where $|\Sigma|$ denotes the determinant of $\Sigma$, $\star$ denotes convolution, and
\begin{equation}
    k_{\Sigma}(\bmx)=\frac{1}{\pi^{\frac{d}{2}}|\Sigma|} e^{-\bmx^T\Sigma^{-2}\bmx}.
\end{equation}
From these formulations, as long as $f$ is nice enough, the gradient can be computed~as
\begin{align}
\label{eqn:gradient_equation}
    \nabla g_{\Sigma}(\bmx)
    &=\frac{2}{\pi^{\frac{d}{2}}}\Sigma^{-1}\int_{\mathbb{R}^d}\bmu g\left(\bmx+\Sigma\bmu\right)e^{-\|\bmu\|^2}\;d\bmu,
\end{align}
where we first pass the gradient into the integral~\eqref{eqn:change_of_variables} and the take the gradient of $k_{\Sigma}$.
Appendix~\ref{app:background_results} provides the detailed derivation of the above results.
In the stochastic case, we write the Gaussian smoothed version of $f_{k}$ as $f_{k,\Sigma}$ to indicate that we first sample and then smooth.
If each $f_k$ is bounded below, then we have that\footnote{This is due to the Fubini-Tonelli theorem (Theorem 8.8~\cite{rudin}), which allows us to switch the order of the integrals.}
\begin{equation}
    f_{\Sigma}(\bmx)=\mathbb{E}\big[f_{k,\Sigma}(\bmx)\big].
\end{equation}
In this setting, instead of solving~\eqref{eqn:min_of_f}, we want to find a solution of
\begin{equation}
    \min_{\bmx,\Sigma}\left\{
        f_{\Sigma}(\bmx)
        :\bmx\in\mathbb{R}^d,\Sigma\text{ symmetric and invertible}
    \right\},
\end{equation}
which occurs when $\Sigma=0$ (see Lemma~\ref{lem:smoothing_maintains_properties_gsgd} (\ref{item:lemma_f})).


\begin{figure}
    \centering
    \begin{minipage}{0.33\textwidth}
        \begin{subfigure}[t]{\textwidth}
            \centering
            \includegraphics[width=\linewidth]{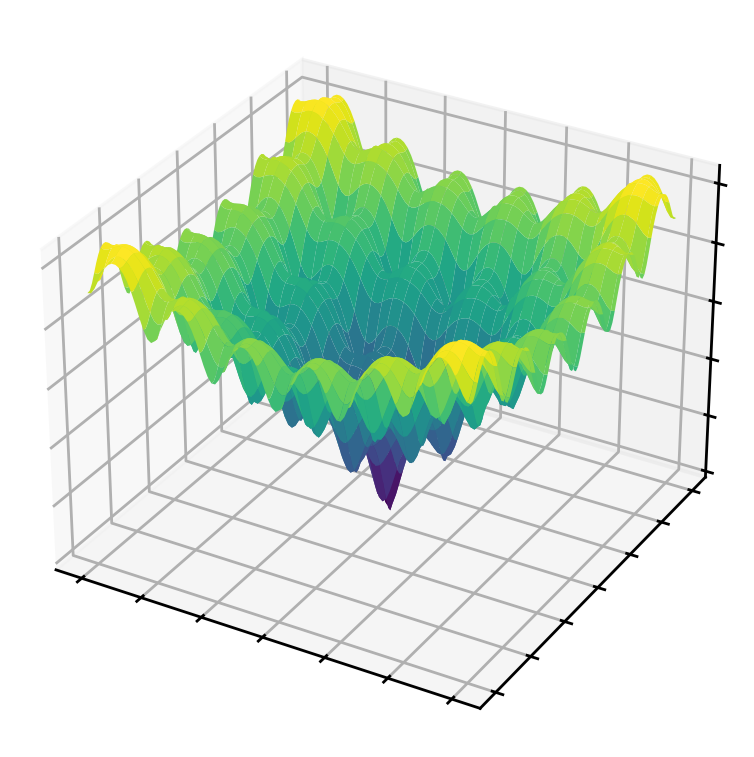}
            \caption{Ackley Function}
            \label{fig:example_anisotropic_smoothing_1}
        \end{subfigure}
    \end{minipage}
    \begin{minipage}{0.66\textwidth}
        \begin{subfigure}[t]{0.49\textwidth}
            \centering
            \includegraphics[width=\linewidth]{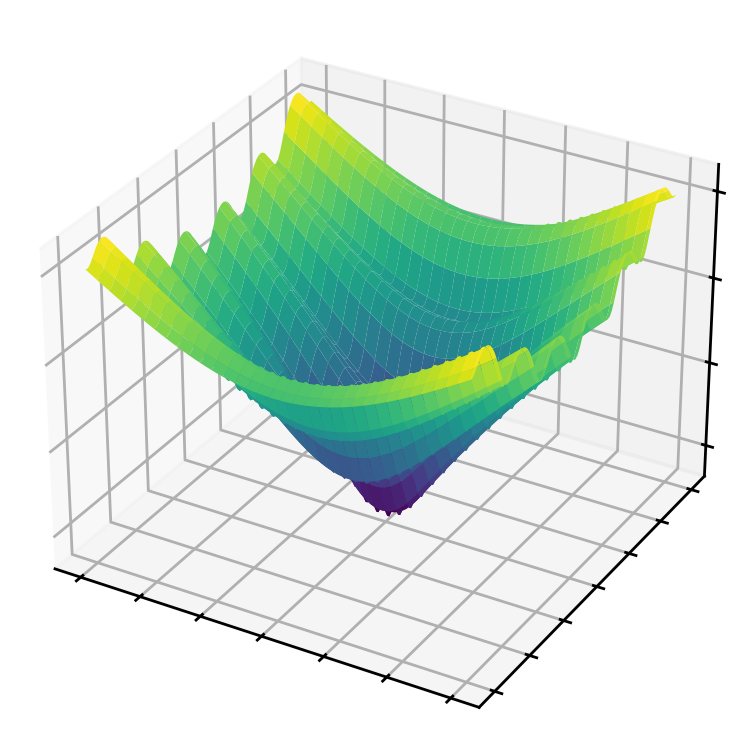}
            \caption{$\bm{e_1}$}
            \label{fig:example_anisotropic_smoothing_2}
        \end{subfigure}
        \begin{subfigure}[t]{0.49\textwidth}
            \centering
            \includegraphics[width=\linewidth]{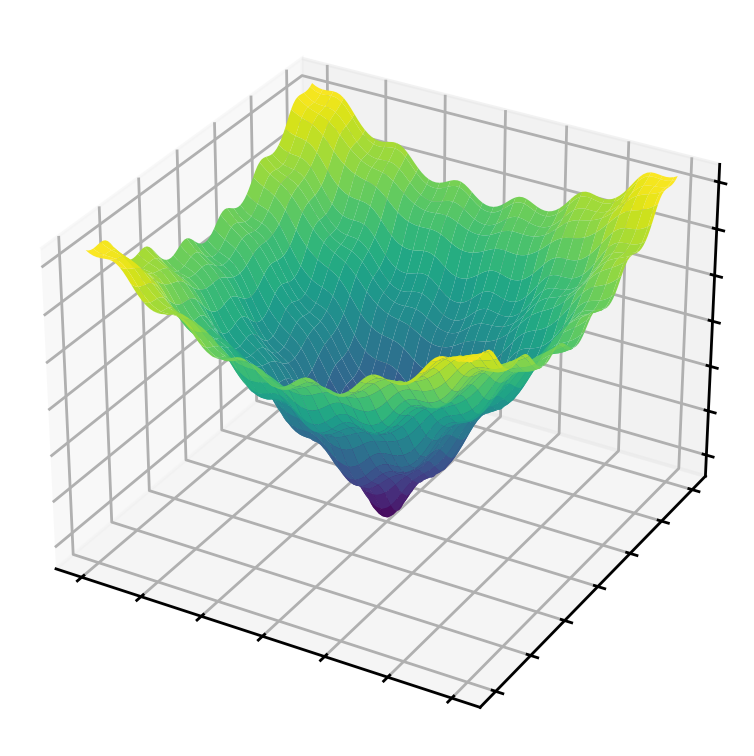}
            \caption{$\bm{e}_1+\bm{e}_2$}
            \label{fig:example_anisotropic_smoothing_3}
        \end{subfigure}
        \begin{subfigure}[t]{0.49\textwidth}
            \centering
            \includegraphics[width=\linewidth]{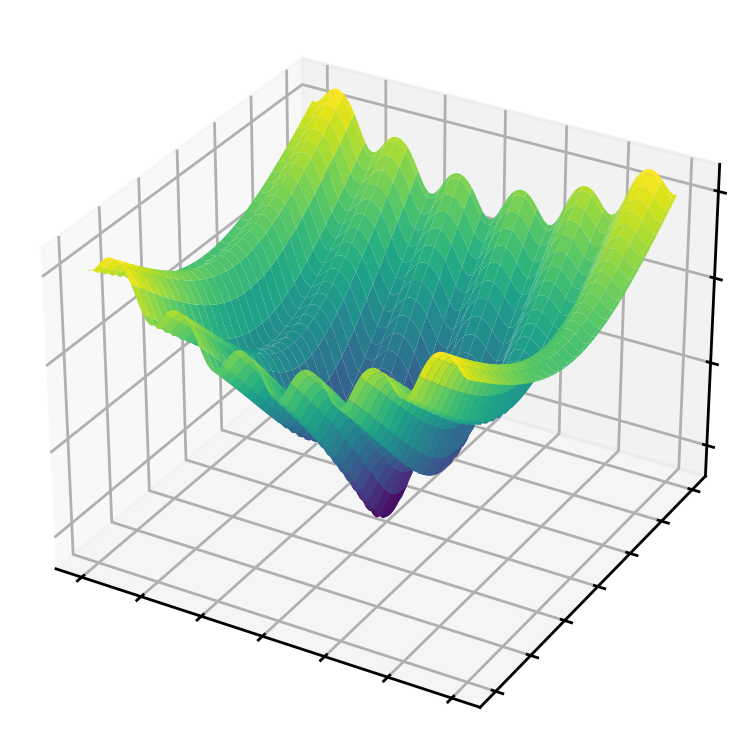}
            \caption{$\bm{e}_2$}
            \label{fig:example_anisotropic_smoothing_4}
        \end{subfigure}
        \begin{subfigure}[t]{0.49\textwidth}
            \centering
            \includegraphics[width=\linewidth]{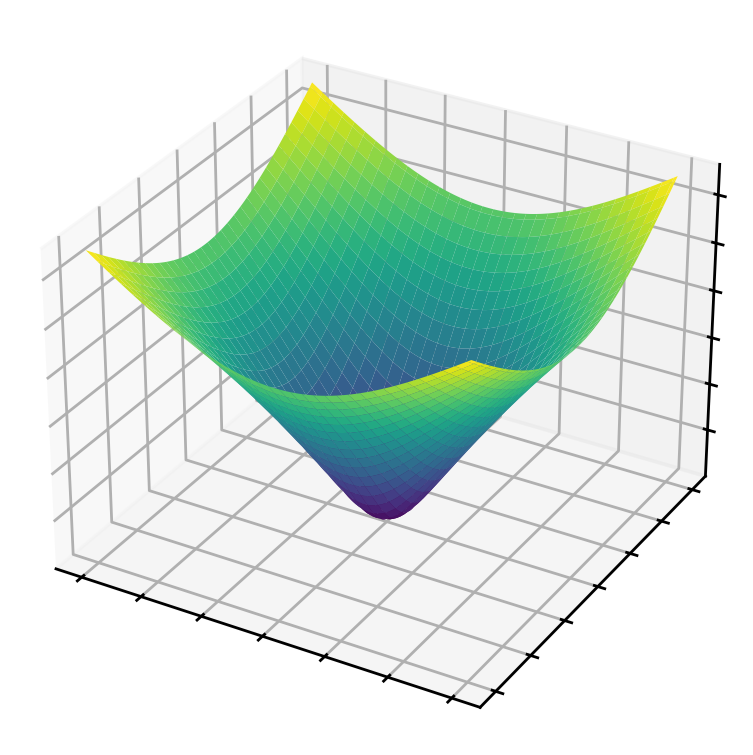}
            \caption{$\bm{e}_1$ and $\bm{e}_2$}
            \label{fig:example_anisotropic_smoothing_5}
        \end{subfigure}
    \end{minipage}
    \caption{
        Anisotropic Gaussian smoothing applied to the Ackley function, in the directions indicated, where $\{\bm{e}_1,\bm{e}_2\}$ is the canonical basis of $\mathbb{R}^2$.
        The smoothing values along the indicated directions are large enough so that the function is convex in the smoothed direction (e.g., in (\ref{fig:example_anisotropic_smoothing_2}), for fixed $y$, the function is convex in $x$).
        When smoothing in both directions (i.e., (\ref{fig:example_anisotropic_smoothing_5})), the smoothed function is fully convex.
    }
    \label{fig:example_anisotropic_smoothing}
\end{figure}

A Gaussian smoothed function is differentiable even if the original function is not~\cite{bertsekas1973stochastic}, and it can be convex even if the underlying function is far from being convex~\cite{mobahi2015link}.
Gaussian smoothing ensures that Lipschitz functions have Lipschitz gradients~\cite{NesterovSpokoiny15}.
Additionally, for certain functions and sufficiently large smoothing values, the Gaussian smoothed function may have a known minimizer~\cite{mobahi2012optimization,loog2001behavior}.
Some examples of smoothing the Ackley function can be seen in Figure~\ref{fig:example_anisotropic_smoothing}.

We generalize GD by starting with a sequence of symmetric, invertible matrices, $(\Sigma_t)_{t=1}^{T}$, and then use the gradient of $f_{\Sigma_t}$ rather than just the gradient of $f$.
We call these $\Sigma_t$ smoothing matrices.
Other works have considered similar modifications of GD.
In~\cite{NesterovSpokoiny15}, $\Sigma_t$ is the same for all $t$.
This is generalized in~\cite{gsgd}, where $\Sigma_t=\sigma_t I_d$ for some sequence $(\sigma_t)_{t=1}^{T}$ (where $I_d$ is the $d$ dimensional identity matrix).
This means that the smoothing directions are the same every iteration, just the amount of smoothing is changed.
As such, we refer to these two methods as isotropic Gaussian smoothing.
If $\sigma>0$, then we use the notation $f_{\sigma}$ to denote $f_{\sigma I_d}$.
In our case, the eigenvectors of $\Sigma_t$ can change with each iteration, which means that as we iterate we can focus on directions that need to be smoothed more than others; hence the anisotropic moniker.
The full algorithmic details of our anisotropic generalizations of GD, SGD, and Adam can be found in Algorithms~\ref{alg:agsgd},~\ref{alg:agssgd}, and~\ref{alg:agsadam}, respectively.

To conclude the introduction, we mention some notation that we use;
for $K\in\mathbb{N}$, $[K]=\{1,2,...,K\}$ and the uniform distribution on $[K]$ is denoted $\text{Unif}[K]$.


\subsection{Related Works}
\label{sec:related_works}

Gaussian smoothing has a long history in optimization~\cite{blake1987visual,more1996smoothing,duchi2012randomized,hazan2016graduated}.
The seminal paper about Gaussian smoothing is~\cite{NesterovSpokoiny15}, which proves many fundamental results (e.g., Lemma~\ref{lem:smoothing_maintains_properties_nesterov}) and uses Gaussian smoothing as a gradient free optimization method.
Their paper fixes the smoothing value, and we extend their results by introducing a changing smoothing matrices.
The series of papers~\cite{gsgd,gssgd} provide isotropic versions of Algorithms~\ref{alg:agsgd},~\ref{alg:agssgd}, and~\ref{alg:agsadam}, denoted GSmoothGD, GSmoothSGD, and GSmoothAdam, respectively.
Many of our results are based on these works.
The key difference is that, just like in~\cite{NesterovSpokoiny15}, they fix an initial Gaussian distribution and then they focus on a smoothing $\Sigma_t=\sigma_tI_d$ for a sequence $(\sigma_t)_{t=1}^{T}\subseteq[0,\infty)$.
We do not have an underlying, stationary distribution, but rather (potentially) change it each iteration.
For an overview of works related to Gaussian smoothing, see Section 1.2 of~\cite{gsgd}.

Another particular relevant paper on Gaussian smoothing, especially regarding how to update the smoothing parameter, is~\cite{iwakiri2022single}.
They provide a particular instance of GSmooth(S)GD where the smoothing parameter, $\sigma_t$, either follows a geometric progression or is updated based on the heat equation (since $f_{\sigma_t}=f\star k_{\sigma_tI_d}$ is the fundamental solution of the heat equation).
Depending on $\Sigma_t$, we will not be able to apply a similar heat equation based update since $f_{\Sigma}$ may not be a solution to the heat equation. Instead, we will use the Covariance Matrix Adaptation (CMA) update.

CMA is a state-of-the-art stochastic optimization algorithm for continuous domains \cite{Hansen2001}. It adapts the covariance matrix of a multivariate normal distribution to efficiently explore the parameter space, making it particularly effective for non-linear, non-convex optimization problems. CMA is self-adaptive, requiring minimal parameter tuning, and has demonstrated robust performance across various domains, including hyperparameter optimization, robotics, and deep learning \cite{Hansen2006,Loshchilov2017}. The algorithm iteratively samples candidate solutions, evaluates their fitness, and updates the distribution parameters based on successful candidates, gradually improving the search distribution \cite{Hansen2004}. CMA has been successfully applied to high-dimensional problems and has shown superior performance compared to many other evolutionary algorithms and gradient-based methods \cite{Auger2005,Igel2007}.
Note that the CMA update is used in our experiments (see Section~\ref{sec:numerical_methods}), but our theory is applicable to any sequence of smoothing matrices regardless of update method.

As mentioned before, our method lends itself to being a gradient free method by numerically approximating \eqref{eqn:gradient_equation} (typically using Monte Carlo).
See the discussion in Section~\ref{sec:numerical_methods} for details.
Another zero-order, computationally efficient modification of Gaussian smoothing is given by~\cite{zhang2020scalable,tran2020adadgs}.
Instead of approximating $\nabla f_{\Sigma}$ directly, they approximate $\nabla f$ by first Gaussian smoothing in $d$ orthogonal directions and then computing the analogous one-dimensional directional gradients using a quadrature scheme.
Many gradient free methods use a similar approach, but their motivations do not come from Gaussian smoothing.
A good overview of zero-order methods is~\cite{liu2020primer}, where several, but not all, of the methods discussed there are related to Gaussian smoothing.

\section{Background and preliminary results}

In this section, we provide the necessary results and definitions required to prove the convergence of our anisotropic Gaussian smoothing algorithms.
As is typically the case with optimization results, we occasionally assume that $f$ is convex and often assume that $f$ is $L$-smooth; $f$ being $L$-smooth means that $\nabla f$ is $L$-Lipschitz.
Let $\Sigma\in\mathbb{R}^{d\times d}$ be invertible, then we denote the operator norm of $\Sigma$ as
\begin{equation}
    \|\Sigma\| = \max_{\|x\|\leq 1}\|\Sigma x\|.
\end{equation}
This means that $\|\Sigma\|$ is the value of the largest eigenvalue.

Our goal is to show that our anisotropic Gaussian smoothing algorithms converge by mimicking the proofs in the unsmoothed (this was done for the isotropic case in~\cite{gsgd,gssgd}).
In order to do this, we need to use a few results from~\cite{NesterovSpokoiny15} that show that applying Gaussian smoothing to a function preserves certain properties.

\begin{lem}
\label{lem:smoothing_maintains_properties_nesterov}
Let $f:\mathbb{R}^d\to\mathbb{R}$ and $\Sigma\in\mathbb{R}^{d\times d}$ be symmetric and invertible.
\begin{enumerate}[(a)]
    \item\label{item:lemma_a}  If $f$ is $L$-smooth then $f_{\Sigma}$ is also $L$-smooth.
    \item\label{item:lemma_b}\label{item:lemma_e}  If $f$ is convex, then $f_{\Sigma}$ is also convex and $f_{\Sigma}(x)\geq f(x)$.
    \item\label{item:lemma_c} If $f$ is $M$-Lipschitz then $f_{\Sigma}$ is $M$-Lipschitz and $M\sqrt{2d}\|\Sigma^{-1}\|$-smooth.
\end{enumerate}
\end{lem}
In~\cite{NesterovSpokoiny15}, the norm used is $\|\bmx\|_{B}=\langle B\bmx,\bmx\rangle$ for a symmetric, positive definite matrix $B$.
Using this norm instead of the standard Euclidean norm we use here makes the notation simpler, but to some extent hides the role that $B$ plays.
Since we will be changing our smoothing matrix over iterations of the algorithm, we need to explicitly show the role that the linear transformation plays (i.e., we use $B=I_d$).
For example, compared to Lemma 2 of~\cite{NesterovSpokoiny15}, $\|\Sigma^{-1}\|$ takes the role of $\frac{1}{\mu}$ and $B^{-1}$ to explicitly show how the smoothing factor and the linear transformation impacts the smoothness.
Since the statement of this lemma is mostly the same as the results in~\cite{NesterovSpokoiny15}, we only include the proof that requires fundamental modification, specifically part (\ref{item:lemma_c}), in Appendix~\ref{app:background_results}.

Two additional results from~\cite{gsgd} show how the smoothed version of a function relates to the bounds of the original function.
The proofs do not need any changes other than in notation, so we refer the reader to~\cite{gsgd} for the proofs.

\begin{lem}
\label{lem:smoothing_maintains_properties_gsgd}
Let $f:\mathbb{R}^d\to\mathbb{R}$ and $\Sigma\in\mathbb{R}^{d\times d}$ be symmetric and invertible.
\begin{enumerate}[(a)]
    \item\label{item:lemma_f} Suppose $f^*=\inf_{\bmy\in\mathbb{R}^d}f(\bmy)>-\infty$. If $f_{\Sigma}(x_0)=f^*$ for some $x_0\in\mathbb{R}^d$, then $f$ is constant.
    \item\label{item:lemma_d} If $f$ is non-constant and either $f(\bmx)\geq m$ or $f(\bmx)\leq M$, then $f_{\Sigma}(\bmx)>m$ or $f_{\Sigma}(\bmx)<M$, respectively.
\end{enumerate}
\end{lem}

We now focus on how smoothing changes the function's output and gradient.
These results come from~\cite{NesterovSpokoiny15}, but as discussed before are written to include the full impact of $\Sigma$.
The proofs require a bit of care since we cannot factor out the $\Sigma$-terms but we could pull out terms of $\sigma>0$.
As such, we include the proofs of both parts in the Appendix~\ref{app:background_results}.

\begin{lem}
\label{lem:difference_f_and_fsigma}
If $f$ is $L$-smooth, then
\begin{enumerate}[(a)]
    \item\label{item:smooth_lemma_a} $|f_{\Sigma}(x)-f(x)|\leq\frac{Ld}{4}\|\Sigma\|^2$
    \item\label{item:smooth_lemma_b} $\|\nabla f_{\Sigma}(x)-\nabla f(x)\|\leq L\|\Sigma\|^2\|\Sigma^{-1}\|(\frac{3+d}{2})^{\frac{3}{2}}$
\end{enumerate}
\end{lem}

Note that if $\Sigma=\sigma I_d$, then for (\ref{item:smooth_lemma_a}) $\|\Sigma\|^2=\sigma^2$ and for (\ref{item:smooth_lemma_b}) $\|\Sigma\|^2\|\Sigma^{-1}\|=\frac{1}{\sigma}$, so in both cases we recover the bound from~\cite{NesterovSpokoiny15}.
The final relationship between $\nabla f$ and $\nabla f_{\Sigma}$ we need follows immediately from part (\ref{item:smooth_lemma_b}) of the previous lemma.

\begin{cor}
\label{cor:nesterov_lemma_4_reversed}
\label{lem:nesterov_lemma_4}
Let $\Sigma$ be invertible and $f:\mathbb{R}^d\to\mathbb{R}$ be $L$-smooth.
Then
\begin{equation}
    \|\nabla f(\bmx)\|^2
    \leq 2\|\nabla f_{\Sigma}(\bmx)\|^2
    +\frac{L^2\|\Sigma\|^4\|\Sigma^{-1}\|^2(3+d)^3}{4}.
\end{equation}
The proof of this result can be found in Appendix~\ref{app:background_results} and also holds if we switch the roles of $f$ and $f_{\Sigma}$.
\end{cor}

This result is similar to Lemma 4 of~\cite{NesterovSpokoiny15}.
The primary difference is that the role $\Sigma^{-1}$ plays is explicitly shown here and in~\cite{NesterovSpokoiny15} is veiled in the norm.
As with the previous lemma, if $\Sigma=\sigma I_d$, then $\|\Sigma\|^4\|\Sigma^{-1}\|^2=\sigma^2$ and we again have the bound from~\cite{NesterovSpokoiny15}.

Finally, we focus on switching our smoothing matrix.
In the isotropic case, if $\tau>\sigma>0$, then $f_{\tau}=(f_{\sigma})_{\eta}$ where $\eta=\sqrt{\tau^2-\sigma^2}$.
This means we can use bounds from the previous results where $f_{\sigma}$ replaces $f$ and we smooth by $\eta$ (see~\cite{gsgd} Lemma 2.7).
In the anisotropic case, given symmetric, invertible $\Sigma$ and $\Tau$, it is not necessarily the case that we can find a symmetric, invertible $\Eta$ so that $(f_{\Sigma})_{H}=f_{\Tau}$ or vice-versa.
In particular, we know if $(f_{\Sigma})_\Eta=f_{\Tau}$, then $\Tau^2=\Sigma^2+\Eta^2$.
Depending on $\Sigma$ and $\Tau$, no such positive definite $\Eta$ may exist.
For example, $\Sigma$ and $\Tau$ could each primarily smooth in directions that are orthogonal to each other.
In this case, we would have to unsmooth and then resmooth, which is not possible.
On the other hand, as shown in part~\ref{item:double_lemma_a}, given $\Sigma$ and $\Tau$, there is always a (unique) matrix $\Eta$ so that $(f_{\Sigma})_{\Tau}=f_{\Eta}$.

\begin{lem}
\label{lem:different_smoothing_values}
Let $\Sigma,\Tau\in\mathbb{R}^{d\times d}$ be symmetric and invertible.
\begin{enumerate}[(a)]
    \item\label{item:double_lemma_a} Let $\Eta=\sqrt{\Sigma^2+\Tau^2}$, then $f_{\Eta}(x)=(f_{\Sigma})_{\Tau}(\bmx)$.
    \item\label{item:double_lemma_b} If $\Sigma$ and $\Tau$ diagonalize together, then
    \begin{equation}
        |f_{\Sigma}(\bmx)-f_{\Tau}(\bmx)|
        \leq\frac{Ld}{4}\left(
            \left(\max_{1\leq i\leq d}(0,\sigma_i-\tau_i)\right)^2
            +\left(\max_{1\leq i\leq d}(0,\tau_i-\sigma_i)\right)^2
        \right).
    \end{equation}
    Otherwise
    \begin{equation}
        |f_{\Sigma}(\bmx)-f_{T}(\bmx)|
        \leq\frac{Ld}{4}\left(
            \|\Sigma^2\|
            +\|T^2\|
            -2\min_{1\leq i\leq d}(\sigma_i,\tau_i)
        \right)
    \end{equation}
    \item\label{item:double_lemma_c} If $\Tau^2-\Sigma^2$ or $\Sigma^2-\Tau^2$ is positive semi-definite, then $|f_{\Sigma}(\bmx)-f_{\Tau}(\bmx)|\leq\frac{Ld}{4}\|\Tau^2-\Sigma^2\|$.
\end{enumerate}
\end{lem}
\noindent This lemma highlights the complexity that comes from smoothing with varying matrices compared to varying scalars and its proof can be found in Appendix~\ref{app:background_results}.
Specifically, parts (\ref{item:double_lemma_b}) and (\ref{item:double_lemma_c}) of the preceding lemma show how much more complicated bounding $|f_{\Sigma}(\bmx)-f_{\Tau}(\bmx)|$ is in the anisotropic case compared to the isotropic setting.

\section{Anisotropic Gaussian smoothing gradient descent}

\begin{algorithm}[tb]
    \caption{Anisotropic Gaussian Smoothing Gradient Descent}
    \label{alg:agsgd}
    \begin{algorithmic}[1]
        \Require $f:\mathbb{R}^d\to\mathbb{R}$, $\bmx_0\in\mathbb{R}^d$, $\lrate>0$
        \Require $\Sigma_t\in\mathbb{R}^{d\times d}$ symmetric, invertible matrices for $t=1,...,T$
        \For{$t=1\to T$}
            \State $\bmx_t = \bmx_{t-1} - \lrate\nabla f_{\Sigma_t}(\bmx_{t-1})$
        \EndFor 
    \end{algorithmic}
\end{algorithm}

We are now able to prove results about the convergence of \agsgd (Algorithm~\ref{alg:agsgd}).
As mentioned before, we adapt the gradient descent rule that uses $\nabla f(\bmx_t)$ so that it uses the gradient of the smoothed function $\nabla f_{\sigma_{t+1}}(\bmx_t)$, that is, we use the update rule
\begin{equation}
    \bmx_{t+1}=\bmx_t-\lambda\nabla f_{\sigma_{t+1}}(\bmx_t).
\end{equation}
As with the standard results for gradient descent, we show convergence results for both convex and non-convex, $L$-smooth functions.
In both of the convergence results, we need to switch from $f_{\sigma_t}$ to $f_{\sigma_{t+1}}$, which requires bounding their difference.
In order to keep the results as general as possible, we denote the bound as $B(\Sigma_{t},\Sigma_{t+1})$, since by Lemma~\ref{lem:different_smoothing_values}, it can vary depending on $\Sigma_t$ and $\Sigma_{t+1}$.
We begin this section with the convergence results for $L$-smooth, convex functions.

\begin{thm}
\label{thm:AGSmoothGD_convex}
Consider \agsgd in Algorithm~\ref{alg:agsgd}.
Let $f:\mathbb{R}^d\to\mathbb{R}$ be convex and $L$-smooth with a minimizer $\xstar$.
Then
\begin{equation}
    f(\bmx_{T})-f(\xstar)
    \leq\frac{1}{2T}\|\bmx_0-\xstar\|^2+\frac{1}{T}\left(
        \frac{Ld}{4}\sum_{t=1}^{T}\|\Sigma_{t}\|^2+\sum_{t=1}^{T-1}tB(\Sigma_{T-t},\Sigma_{T-t+1})
    \right),
\end{equation}
where $B(\Sigma,\Tau)$ satisfies $|f_{\Sigma}(\bmx)-f_{\Tau}(\bmx)|\leq B(\Sigma,\Tau)$.
\end{thm}

If we replace $\Sigma_t$ with 0, we end up with the regular convergence result for gradient descent.
On the other hand, if we replace $\Sigma_t$ with $\sigma_tI_d$, we recover the results from~\cite{gsgd}.
As such, we really have generalized the results from both of these cases.

The proof can be found in Appendix~\ref{app:agsgd_proofs} and is a modification of the proof in~\cite{gsgd}.
The key differences come from the complexity that arises when we switch from $\sigma\geq 0$ to $\Sigma\in\mathbb{R}^{d\times d}$.
The majority of the work required to make this change comes from the results in the previous section (e.g., Lemma~\ref{lem:different_smoothing_values}).

When switching from convex to non-convex functions, we no longer know that smoothing increases the function values (i.e., $f_{\Sigma}(\bmx)\geq f(\bmx)$).
Nevertheless, we can once again mimic the proof in~\cite{gsgd} (which itself mimics the standard gradient descent proof for non-convex functions) in order to see a minimum gradient norm-type convergence.

\begin{thm}
\label{thm:AGSmoothGD_nonconvex}
Consider \agsgd in Algorithm~\ref{alg:agsgd}.
Let $f:\mathbb{R}^d\to\mathbb{R}$ be $L$-smooth with minimum $f^*$.
Then
\begin{multline}
    \min_{t=1,...,T}\|\nabla f(\bmx_t)\|^2
    \leq\frac{4}{T\lrate}\big(f(\bmx_{0})-f^*\big)\\
        +\frac{L^2}{T}\left(\frac{6+d}{2}\right)^3\sum_{t=1}^{T}\|\Sigma_t\|^4\|\Sigma_t^{-1}\|^2
        +\frac{4}{T\lrate}\sum_{t=0}^{T}B(\Sigma_{t+1},\Sigma_{t}),
\end{multline}
where $B(\Sigma,\Tau)$ satisfies $|f_{\Sigma}(\bmx)-f_{\Tau}(\bmx)|\leq B(\Sigma,\Tau)$.
\end{thm}
As with the previous theorem, we can recover the results from unsmoothed gradient descent ($\Sigma_t$ replaced with 0) and~\cite{gsgd} ($\Sigma_t$ replaced with $\sigma_tI_d$).
The proof can be found in Appendix~\ref{app:agsgd_proofs}.

In both theorems of this section, the second term on the right hand sides indicate the inherent cost of smoothing, which is positive whenever $\Sigma_t\not=0$.
The last terms show the cost of varying the smoothing matrices.
In both cases, as $T\to\infty$, the right hand sides both converge to 0 when $(\|\Sigma_t\|)_{t\geq 1}$ is bounded.
While we see the theoretical cost of smoothing, experimentally using smoothing we see faster convergence.

\section{Anisotropic Gaussian smoothing stochastic gradient descent}

\begin{algorithm}[tb]
    \caption{Anisotropic Gaussian Smoothing Stochastic Gradient Descent}
    \label{alg:agssgd}
    \begin{algorithmic}[1]
        \Require $f:\mathbb{R}^d\to\mathbb{R}$, $\bmx_0\in\mathbb{R}^d$, $\lrate>0$
        \Require $\Sigma_t\in\mathbb{R}^{d\times d}$ symmetric, invertible matrices for $t=1,...,T$
        \For{$t=1\to T$}
            \State $k_t\sim\text{Unif}([K])$
            \State $\bmx_t = \bmx_{t-1} - \lrate\nabla f_{k_t,\Sigma_t}(\bmx_{t-1})$
        \EndFor 
    \end{algorithmic}
\end{algorithm}

We now prove the stochastic convergence results for non-convex, $L$-smooth functions.
As with the gradient descent case, for different smoothing values of $\Sigma_t$, we can recover the convergence results for both SGD and GSmoothSGD.
To adapt SGD, we use the update
\begin{equation}
    \bmx_{t+1}=\bmx_t-\lambda_{t+1}\nabla f_{k_{t+1},\sigma_{t+1}}(\bmx_t),
\end{equation}
replacing $\nabla f_{k_{t+1}}(\bmx_t)$ in SGD with its smoothed counterpart.
The full details can be found in Algorithm~\ref{alg:agssgd} and the proof of the following theorem can be found in Appendix~\ref{app:agssgd_proofs}.

\begin{thm}
\label{thm:GSSGD}
Consider Anisotropic GSmoothSGD in Algorithm~\ref{alg:agssgd}.
Assume that $f$ is $L$-smooth in $\bmx$ and $\mathbb{E}\big[\nabla f_k(\bmx)\big]=\nabla f(\bmx)$.
Let $f^*$ denote the minimum of $f$.
Suppose that $\mathbb{E}\big[\|\nabla f_k\|^2\big]\leq\lambda$ for any $k\in [K]$. 
Then for some $\nu<T$, 
\begin{multline}
\label{eqn:thm_gssgd}
    \mathbb{E}\big[\|\nabla f_{\Sigma_{\nu+1}}(\bmx_\nu)\|^2\big]
    \leq
        \frac{f(\bmx_0)-f_*}{\sum_{t=1}^{T}\eta_t}
        +\frac{L\lambda^2\sum_{t=1}^{T}\eta_t^2}{\sum_{t=1}^{T}\eta_t}\\
        +\frac{L^3(3+d)^3}{8}\frac{\sum_{t=1}^{T}\|\Sigma_t\|^4\|\Sigma_t^{-1}\|^2\eta_t^2}{\sum_{t=1}^{T}\eta_t}
        +\frac{\sum_{t=1}^{T}B(\Sigma_{t},\Sigma_{t-1})}{\sum_{t=1}^{T}\eta_t}
\end{multline}
\end{thm}


As $T\to\infty$, if $(\|\Sigma_t\|)_{t\geq 1}$ are bounded and $\sum_{t\geq 1}\eta_t^2<\infty$ but $\sum_{t\geq 1}\eta_t=\infty$, then the right hand side of~\eqref{eqn:thm_gssgd} converges to 0.
The first two terms on the right hand side of~\eqref{eqn:thm_gssgd} are the same as the typical, unsmoothed case.
The additional terms indicate an additional theoretical cost of smoothing.
The first of these terms is an inherent cost of smoothing, which can only be 0 if we do not smooth.
The second term indicates the cost of changing the smoothing parameter.
In particular, if $\Sigma_{t}=\Sigma$ for all $t\geq 1$, then $B(\Sigma_t,\Sigma_{t-1})=0$ for all $t\geq 1$, which entirely removes the last term.

Despite these additional terms, in practice, we see that Gaussian smoothing substantially improves the speed of convergence.
Anecdotally, it seems that Gaussian smoothing helps with SGD's ability to generalize (see numerical experiments in~\cite{gssgd}), which is similar observation to increasing the mini-batch size in SGD~\cite{keskar2016large}.

If we use a constant learning rate ($\eta_t=\eta$ for all $t\geq 1$), then Algorithm~\ref{alg:agssgd} converges to a noisy ball which varies based on the smoothing parameters (i.e., as $t\to\infty$, $\nabla f_{\Sigma_{t+1}}(\bmx_t)$ is in a ball around $\bm{0}$ with radius based on $\lambda$, rather than $\nabla f_{\Sigma_{t+1}}(\bmx_t)\to 0$).
This indicates the need for a variable learning rate, which is the same conclusion one arrives to in the unsmoothed setting.

\section{Anisotropic Gaussian smoothing Adam}

In this section, we adapt the Adam update to use $\nabla f_{k_{t+1},\sigma_{t+1}}(\bmx_t)$ rather than $\nabla f_{k_{t+1}}(\bmx_t)$.
The full details can be found in Algorithm~\ref{alg:agsadam}.
Our last convergence result is for AGS-Adam, which is a generalization of~\cite{he2023convergence} Theorem 11 and its proof is adapted from the proof of the aforementioned result.
The proof along with all of the other modifications can be found in Appendix~\ref{app:agsadam_proofs}.
The generalized result arrives as the same almost sure convergence result as the original proof and unlike the convergence results for \agsgd and AGS-SGD, has no additional terms that indicate any cost of smoothing.
As discussed in~\cite{he2023convergence}, the assumptions we have here are common or even shown to be necessary for stochastic optimization convergence.

\begin{thm}
\label{thm:agsadam}
Consider \agsadam in Algorithm~\ref{alg:agsadam}.
Let $f$ be $L$-smooth and $f^*$ denote the minimum of $f$.
Assume $\mathbb{E}\big[f_k(\bmx)\big]=f(\bmx)$ and $\mathbb{E}\big[\nabla f_k(\bmx)\big]=\nabla f(\bmx)$.
Suppose that $\mathbb{E}\big[\|\nabla f_k\|^2\big]\leq\lambda$ for any $k\in[K]$,
\begin{equation}
    \sum_{t=1}^{\infty}\eta_t=\infty,
    \quad\sum_{t=1}^{\infty}\eta_t^2<\infty,
    \quad\sum_{t=1}^{\infty}\eta_t(1-\theta_t)<\infty,
\end{equation}
and there exists a non-increasing sequence of real numbers $(\alpha_t)_{t\geq 1}$ such that $\eta_t=\Theta(\alpha_t)$.
Let $\widetilde{B}$ satisfy
\begin{equation}
\label{eqn:agsadam_grad_bound}
    \widetilde{B}(\Sigma,\Tau)
    \geq\|\nabla f_{\Sigma}(x)-\nabla f_{\Tau}(x)\|.
\end{equation}
Suppose for some $\widetilde{M}>0$ that
\begin{equation}
    \widetilde{B}(\Sigma_{t+1},\Sigma_t)\leq \widetilde{M}\eta_t,\;
    \sum_{t=1}^{\infty}B(\Sigma_t,\Sigma_{t-1})<\infty,
    \text{ and }
    \|\Sigma_t\|\to 0.
\end{equation}
Then
\begin{equation}
    \lim_{t\to\infty}\|\nabla f(\bmx_t)\|^2=0\text{ a.s.}
    \text{ and }
    \lim_{t\to\infty}\mathbb{E}\big[\|\nabla f(\bmx_t)\|^2\big]=0.
\end{equation}
\end{thm}

The $\widetilde{B}$ term is guaranteed to exist by Lemma~\ref{lem:difference_f_and_fsigma} (\ref{item:smooth_lemma_b}), specifically
\begin{align}
    \|\nabla f_{\Sigma}(x)-\nabla f_{\Tau}(x)\|
    &\leq
        \|\nabla f_{\Sigma}(x)-\nabla f(x)\|
        +\|\nabla f(x)-\nabla f_{\Tau}(x)\|\\
    &\leq L(\tfrac{3+d}{2})^{\frac{3}{2}} (\|\Sigma\|^2\|\Sigma^{-1}\|+\|\Tau\|^2\|\Tau^{-1}\|).
\end{align}
Using this bound, we need $\|\Sigma_t\|^2\|\Sigma_t^{-1}\|\leq\widetilde{M}\eta_t$, which actually promises $\|\Sigma_t\|\to 0$ since $\eta_t\to 0$.
However, this bound can be improved upon (see, e.g.,~\cite{gssgd} Lemma 2.2).
Similarly, as with Theorems~\ref{thm:AGSmoothGD_convex},~\ref{thm:AGSmoothGD_nonconvex}, and~\ref{thm:GSSGD}, $B(\Sigma,\Tau)$ varies depending on $\Sigma$ and $\Tau$ (see Lemma~\ref{lem:different_smoothing_values}).
As such, in order to provide the most general result we do not specify either $B$ or $\widetilde{B}$.

\begin{algorithm}[tb]
    \caption{Anisotropic Gaussian Smoothing Adam}
    \label{alg:agsadam}
    \begin{algorithmic}[1]
        \Require{$\eta_t>0$, $\epsilon\geq 0$, $0\leq\beta_t\leq\beta<1$, $\theta_t\in(0,1)$, $x_0\in\mathbb{R}^d$}
        \Require{$\Sigma_t$ symmetric, positive definite for $t=1,2,...$}
        \Ensure{$m_0=0$, $v_0=0$}
        \For{$t=1,2,...$}
            \State $k_t\sim\text{Unif}([K])$
            \State $\bmm_{t}=\beta_t\bmm_{t-1}+(1-\beta_t)\nabla f_{k_t,\Sigma_t}(\bmx_{t-1})$
            \State $\bmv_t=\theta_t\bmv_{t-1}+(1-\theta_t)(\nabla f_{k_t,\Sigma_t}(\bmx_{t-1}))^2$
            \State $\bmx_t=\bmx_{t-1}-\eta_t\dfrac{\bmm_{t}}{\sqrt{\bmv_t+\epsilon}}$
        \EndFor
    \end{algorithmic}
\end{algorithm}


\section{Numerical Methods and Experiments}
\label{sec:numerical_methods}

We now shift our focus to discussing numerical evidence that shows the behavior of anisotropic Gaussian smoothing.
Explicitly calculating $f_{\Sigma}$ or $\nabla f_{\Sigma}$ is difficult, if not impossible in the majority of cases.
As such, we can use~\eqref{eqn:definition_of_fsigma} or~\eqref{eqn:gradient_equation}, respectively, in order to approximate the value.
First, for any $\bmx\in\mathbb{R}^d$,
\begin{equation}
    \Sigma^{-1}\int_{\mathbb{R}^d}\bmu f(\bmx)e^{-\|\bmu\|^2}\;d\bmu
    =0.
\end{equation}
Second, using a simple change of variables,
\begin{equation}
    \Sigma^{-1}\int_{\mathbb{R}^d}\bmu f(\bmx+\Sigma\bmu)e^{-\|\bmu\|^2}\;d\bmu
    =-\Sigma^{-1}\int_{\mathbb{R}^d}\bmu f(\bmx-\Sigma\bmu)e^{-\|\bmu\|^2}\;d\bmu.
\end{equation}
This means that if $\delta_{\Sigma}(\bmx,\bmu)$ is either of the following
\begin{equation}
    \frac{f(\bmx+\Sigma\bmu)-f(\bmx)}{2}
    \text{ or }
    f(\bmx+\Sigma\bmu)-f(\bmx-\Sigma\bmu),
\end{equation}
then
\begin{equation}
\label{eqn:zero_order_grad_estimate}
    \nabla f_{\Sigma}(\bmx)
    \approx\frac{1}{N}\sum_{n=1}^{N}\delta_{\Sigma}(\bmx,\bmu_n)\Sigma^{-1}\bmu_n
\end{equation}
where $\bmu_n$ is sampled from the $d$-dimensional standard normal distribution.
We can now replace $\nabla f_{\Sigma_t}$ in any of our anisotropic smoothing algorithms with its Monte Carlo estimate in~\eqref{eqn:zero_order_grad_estimate} to obtain a zero-order method.
This is the method we use in our numerical experiments.
This type of Monte Carlo estimate plays a key role in many zero-order optimization methods (see~\cite{liu2020primer} for an overview).
However, our anisotropic approach to updating $\Sigma$ is novel in the context of Gaussian smoothing.

Here we provide the definition of the 5 benchmark functions.
To make the test functions more general, we applied the following linear transformation to $\bm x$, i.e.,
\[
\bm z = \mathbf{R}(\bm x - \bm x_{\rm opt}),
\]
where $\mathbf{R}$ is a rotation matrix making the functions non-separable and $\bm x_{\rm opt}$ is the optimal state. Then we substitute $\bm z$ into the standard definitions of the benchmark functions to formulate our test problems. For notational simplicity, we use $\bm z$ as the input variable in the following definitions and omit the dependence of $\bm z$ on $\bm x$.
Lastly, we update $\Sigma_t$ using the CMA-update as discussed in Section~\ref{sec:related_works}.
Results for the experiments can be found in Figure~\ref{fig:numeric_examples}
\vspace{0.1cm}
\begin{itemize}\itemsep0.3cm

\item $F_1(\bm x)$: {\bf Sphere} function
\begin{equation*}
F_1(\bm x) = {\sum_{i=1}^d x_i^2},
\end{equation*}
where $d$ is the dimension and $\bm{x} \in [-2, 2]^d$ is the input domain. The global minimum is $f(\bm{x}_{\rm opt}) = 0$. This represents {\em convex} landscapes. 
Results shown in Figure~\ref{fig:example_anisotropic_smoothing_1}.

\item $F_2(\bm x)$: {\bf Ellipsoidal} function
\begin{equation*}
F_2(\bm x) = {\sum_{i=1}^d 10^{6\frac{i-1}{d-1}} x_i^2},
\end{equation*}
where $d$ is the dimension and $\bm{x} \in [-2, 2]^d$ is the input domain. The global minimum is $f(\bm{x}_{\rm opt}) = 0$. This represents {\em convex and highly ill-conditioned} landscapes. 
Results shown in Figure~\ref{fig:example_anisotropic_smoothing_2}.

\item $F_3(\bm x)$: {\bf Different Powers} function $F_3(\bm x)$ is defined by
\begin{equation*}
F_3(\bm x) =  \sqrt{\sum_{i=1}^d |x_i|^{2 + 4\frac{i-1}{d-1}}},
\end{equation*}
where $d$ is the dimension and $\bm{x} \in [-5, 5]^d$ is the input domain. The global minimum is $f(\bm{x}_{\rm opt}) = 0$ at $\bm{x}_{\rm opt} = (0,...,0)$. This represents {\em convex} landscapes with {\em small solution volume}.
Results shown in Figure~\ref{fig:example_anisotropic_smoothing_3}.

\item $F_4(\bm x)$: {\bf Powell} function $F_{18}(\bm x)$ is defined by
\begin{align*}
F_{18}(\bm x) = &\sum_{i=1}^{d/4} [(x_{4i-3} + 10x_{4i-2})^2 + 5(x_{4i-1} - x_{4i})^2 
\\
& + (x_{4i-2} - 2x_{4i-1})^4 + 10  (x_{4i-3} - x_{4i})^4],
\end{align*}
where $\bm{x} \in [-4, 5]^d$ is the input domain. The global minimum is $f(\bm{x}_{\rm opt}) = 0$, at $\bm{x}_{\rm opt} = (0,\cdots,0)$. This function represents {\em convex} landscapes, where {\em the area that contains global minimum are very small.}
Results shown in Figure~\ref{fig:example_anisotropic_smoothing_4}.

\item $F_5(\bm x)$: {\bf Rosenbrock} function
\begin{equation*}
F_6(\bm x) =  \sum_{i=1}^{d-1} [100(z_{i+1}-x_i^2)^2 + (x_i-1)^2],
\end{equation*}
where $d$ is the dimension and $\bm{x} \in [-5, 10]^d$ is the initial search domain. The global minimum is $f(\bm{x}_{\rm opt}) = 0$. The function is {\em unimodal}, and the global minimum lies in a {\em bending ridge}, which needs to be followed to reach solution. The ridge changes its orientation $d - 1$ times. 
Results shown in Figure~\ref{fig:example_anisotropic_smoothing_5}.

\end{itemize}

\begin{figure}[ht!]
    \centering
    \begin{subfigure}[t]{0.46\textwidth}
        \centering
        \includegraphics[width=\linewidth]{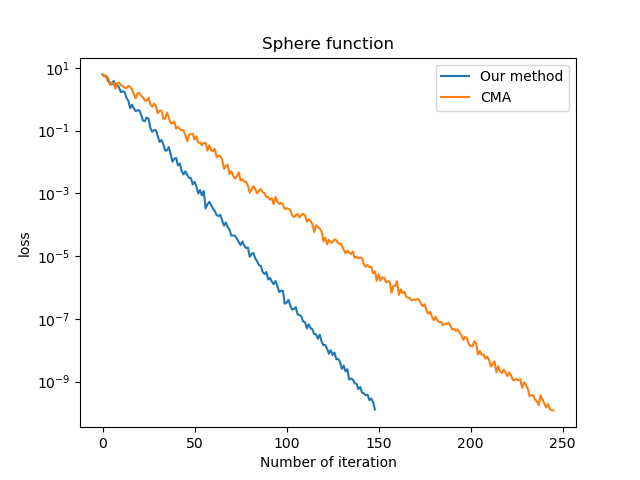}
        \caption{Sphere}
        \label{fig:sphere}
    \end{subfigure}\hfill
    \begin{subfigure}[t]{0.46\textwidth}
        \centering
        \includegraphics[width=\linewidth]{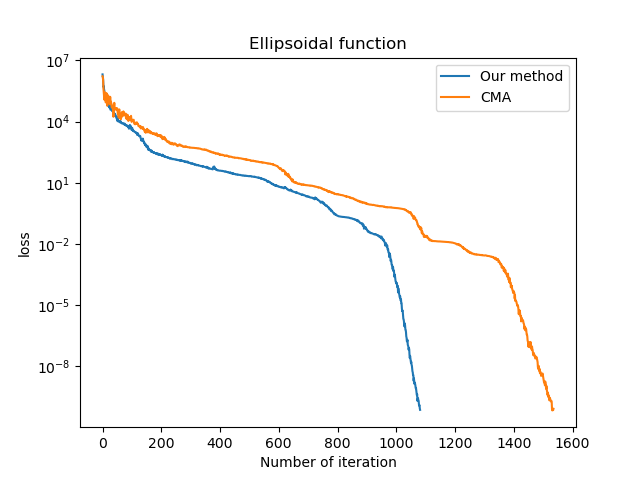}
        \caption{Ellipsoidal}
        \label{fig:ellip}
    \end{subfigure}

    \begin{subfigure}[t]{0.46\textwidth}
        \centering
        \includegraphics[width=\linewidth]{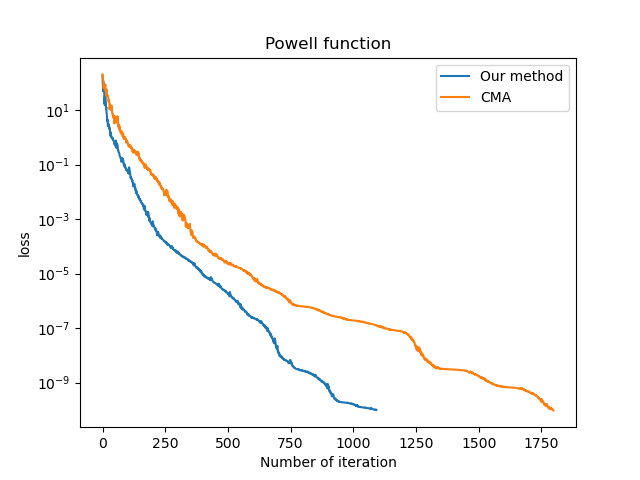}
        \caption{Powell}
        \label{fig:powell}
    \end{subfigure}\hfill
    \begin{subfigure}[t]{0.46\textwidth}
        \centering
        \includegraphics[width=\linewidth]{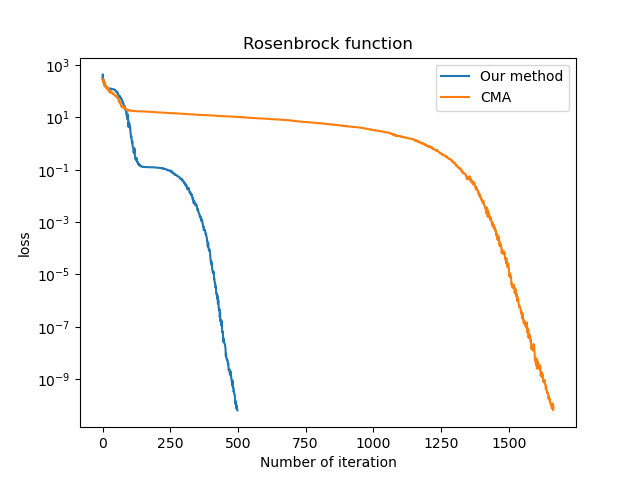}
        \caption{Rosenbrock}
        \label{fig:rosen}
    \end{subfigure}
    
    \begin{subfigure}[t]{0.46\textwidth}
        \centering
        \includegraphics[width=\linewidth]{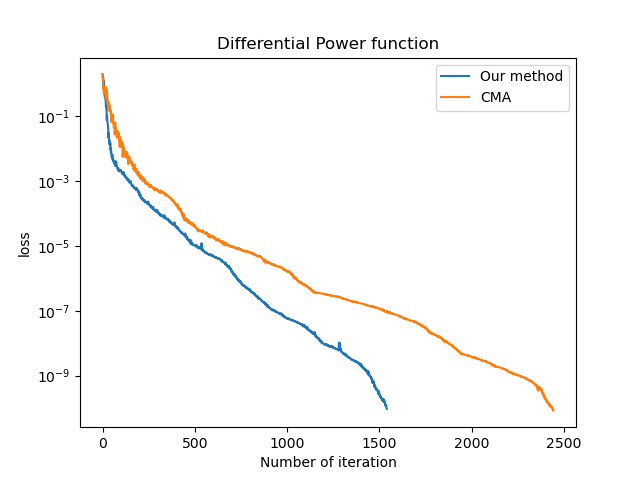}
        \caption{Differential power}
        \label{fig:diff}
    \end{subfigure}
    \caption{Comparison between CMA and AGS-GD for specified functions.}
    \label{fig:numeric_examples}
\end{figure}

\section{Conclusion}

In this effort we introduced a family of anisotropic Gaussian smoothing algorithms, AGS-GD, AGS-SGD, and AGS-Adam, as alternatives to traditional gradient descent optimization methods. These algorithms address a major challenge in optimization: the tendency of methods to become trapped in suboptimal local minima. By replacing the standard gradient with a non-local gradient derived through anisotropic Gaussian smoothing, we provide a more effective path toward convergence. Our convergence analyses build on and extend existing results for standard gradient descent and isotropic Gaussian smoothing, demonstrating their applicability to both convex and non-convex, L-smooth functions. In the stochastic setting, we showed that these algorithms converge to a noisy ball around a minimum, with the size of this ball determined by the smoothing parameters—emphasizing the importance of parameter selection in practice.

While the theoretical benefits of anisotropic Gaussian smoothing are clear, practical implementation remains a challenge due to the computational complexity of calculating smoothed functions or their gradients. To mitigate this, we utilize Monte Carlo estimation to approximate the smoothed gradient, aligning with established zero-order optimization methods and making the AGS algorithms more feasible for real-world applications.

Our work primarily establishes the theoretical foundation and convergence properties of these anisotropic Gaussian smoothing algorithms. 
One particularly useful research direction would be investigating the relationship between smoothing parameter selection and algorithm performance across different problem domains.
We believe that AGS algorithms have significant potential to advance optimization by offering robust and efficient solutions for complex optimization tasks.

\bibliographystyle{plain}
\bibliography{main}

\begin{thebibliography}{10}

\bibitem{Auger2005}
Anne Auger and Nikolaus Hansen.
\newblock A restart {CMA} evolution strategy with increasing population size.
\newblock In {\em 2005 IEEE Congress on Evolutionary Computation}, volume~2, pages 1769--1776. IEEE, 2005.

\bibitem{bertsekas1973stochastic}
Dimitri~P Bertsekas.
\newblock Stochastic optimization problems with nondifferentiable cost functionals.
\newblock {\em Journal of Optimization Theory and Applications}, 12(2):218--231, 1973.

\bibitem{blake1987visual}
Andrew Blake and Andrew Zisserman.
\newblock {\em Visual reconstruction}.
\newblock MIT press, 1987.

\bibitem{duchi2012randomized}
John~C Duchi, Peter~L Bartlett, and Martin~J Wainwright.
\newblock Randomized smoothing for stochastic optimization.
\newblock {\em SIAM Journal on Optimization}, 22(2):674--701, 2012.

\bibitem{Hansen2006}
Nikolaus Hansen.
\newblock The {CMA} evolution strategy: A comparing review.
\newblock {\em Towards a New Evolutionary Computation}, pages 75--102, 2006.

\bibitem{Hansen2004}
Nikolaus Hansen and Stefan Kern.
\newblock Evaluating the {CMA} evolution strategy on multimodal test functions.
\newblock In {\em Parallel Problem Solving from Nature - {PPSN} {VIII}}, pages 282--291, Berlin, Heidelberg, 2004. Springer.

\bibitem{Hansen2001}
Nikolaus Hansen and Andreas Ostermeier.
\newblock Completely derandomized self-adaptation in evolution strategies.
\newblock {\em Evolutionary Computation}, 9(2):159--195, 2001.

\bibitem{hazan2016graduated}
Elad Hazan, Kfir~Yehuda Levy, and Shai Shalev-Shwartz.
\newblock On graduated optimization for stochastic non-convex problems.
\newblock In {\em International conference on machine learning}, pages 1833--1841. PMLR, 2016.

\bibitem{he2023convergence}
Meixuan He, Yuqing Liang, Jinlan Liu, and Dongpo Xu.
\newblock Convergence of adam for non-convex objectives: Relaxed hyperparameters and non-ergodic case.
\newblock {\em arXiv preprint arXiv 2307.11782}, 2023.

\bibitem{Igel2007}
Christian Igel, Nikolaus Hansen, and Stefan Roth.
\newblock Covariance matrix adaptation for multi-objective optimization.
\newblock {\em Evolutionary Computation}, 15(1):1--28, 2007.

\bibitem{iwakiri2022single}
Hidenori Iwakiri, Yuhang Wang, Shinji Ito, and Akiko Takeda.
\newblock Single loop gaussian homotopy method for non-convex optimization.
\newblock {\em Advances in Neural Information Processing Systems}, 35:7065--7076, 2022.

\bibitem{keskar2016large}
Nitish~Shirish Keskar, Dheevatsa Mudigere, Jorge Nocedal, Mikhail Smelyanskiy, and Ping Tak~Peter Tang.
\newblock On large-batch training for deep learning: Generalization gap and sharp minima.
\newblock {\em arXiv preprint arXiv:1609.04836}, 2016.

\bibitem{liu2020primer}
Sijia Liu, Pin-Yu Chen, Bhavya Kailkhura, Gaoyuan Zhang, Alfred~O Hero~III, and Pramod~K Varshney.
\newblock A primer on zeroth-order optimization in signal processing and machine learning: Principals, recent advances, and applications.
\newblock {\em IEEE Signal Processing Magazine}, 37(5):43--54, 2020.

\bibitem{loog2001behavior}
Marco Loog, Johannes JisseDuistermaat, and Luc~MJ Florack.
\newblock On the behavior of spatial critical points under gaussian blurring a folklore theorem and scale-space constraints.
\newblock In {\em Scale-Space and Morphology in Computer Vision: Third International Conference, Scale-Space 2001 Vancouver, Canada, July 7--8, 2001 Proceedings 3}, pages 183--192. Springer, 2001.

\bibitem{Loshchilov2017}
Ilya Loshchilov and Frank Hutter.
\newblock {CMA-ES} for hyperparameter optimization of deep neural networks.
\newblock In {\em International Conference on Learning Representations Workshop}, 2017.

\bibitem{mobahi2012optimization}
Hossein Mobahi.
\newblock {\em Optimization by Gaussian smoothing with application to geometric alignment}.
\newblock University of Illinois at Urbana-Champaign, 2012.

\bibitem{mobahi2015link}
Hossein Mobahi and John~W Fisher.
\newblock On the link between gaussian homotopy continuation and convex envelopes.
\newblock In {\em Energy Minimization Methods in Computer Vision and Pattern Recognition: 10th International Conference, EMMCVPR 2015, Hong Kong, China, January 13-16, 2015. Proceedings 10}, pages 43--56. Springer, 2015.

\bibitem{more1996smoothing}
Jorge~J More and Wu~Zhijun.
\newblock {\em Smoothing techniques for macromolecular global optimization}.
\newblock Springer, 1996.

\bibitem{NesterovSpokoiny15}
Y.~Nesterov and V.~Spokoiny.
\newblock Random gradient-free minimization of convex functions.
\newblock {\em Foundations of Computational Mathematics}, pages 1--40, 2015.

\bibitem{rudin}
Walter Rudin.
\newblock {\em Real \& Complex Analysis}.
\newblock McGraw-Hill, 1987.

\bibitem{gsgd}
Andrew Starnes, Anton Dereventsov, and Clayton Webster.
\newblock Gaussian smoothing gradient descent for minimizing high-dimensional non-convex functions.
\newblock {\em arXiv preprint arXiv 2311.00521}, 2023.

\bibitem{gssgd}
Andrew Starnes and Clayton Webster.
\newblock Improved performance of stochastic gradients with gaussian smoothing.
\newblock {\em arXiv preprint arXiv 2311.00531}, 2024.

\bibitem{tran2020adadgs}
Hoang Tran and Guannan Zhang.
\newblock Adadgs: An adaptive black-box optimization method with a nonlocal directional gaussian smoothing gradient.
\newblock {\em arXiv preprint arXiv:2011.02009}, 2020.

\bibitem{zhang2020scalable}
Jiaxin {Zhang}, Hoang {Tran}, Dan {Lu}, and Guannan {Zhang}.
\newblock {A Scalable Evolution Strategy with Directional Gaussian Smoothing for Blackbox Optimization}.
\newblock {\em {arXiv preprint arXiv:2002.03001}}, 2020.

\end{thebibliography}

\clearpage
\appendix

\section{Proofs of Background Results}
\label{app:background_results}

\begin{proof}[Proof of \eqref{eqn:change_of_variables}]
Let $\bmv=\bmx+\Sigma\bmu$. Since $\Sigma$ is invertible,
\begin{equation}
\label{eqn:cov_1}
    \bmu=\Sigma^{-1}(\bmv-\bmx).
\end{equation}
Note that
\begin{align}
    \bmu_i
    =\big(\Sigma^{-1}(\bmv-\bmx)\big)_i
    =\sum_{k=1}^{d}\Sigma_{i,k}^{-1}(\bmv_k-\bmx_k).
\end{align}
This means $\frac{\partial \bmu_i}{\partial \bmv_j}=\Sigma_{i,j}^{-1}$ and
\begin{align}
\label{eqn:cov_2}
    \frac{\partial\bmu}{\partial\bmv}
    =\left|\begin{array}{ccc}
        \frac{\partial \bmu_1}{\partial \bmv_1}&\cdots&\frac{\partial \bmu_1}{\partial \bmv_d}\\
        \vdots&\ddots&\vdots\\
        \frac{\partial \bmu_d}{\partial \bmv_1}&\cdots&\frac{\partial \bmu_d}{\partial \bmv_d}
    \end{array}\right|
    =|\Sigma^{-1}|
    =\frac{1}{|\Sigma|}.
\end{align}
Also,
\begin{align}
\label{eqn:cov_3}
    \|\bmu\|^2
    =\|\Sigma^{-1}(\bmv-\bmx)\|^2
    =(\bmv-\bmx)^T(\Sigma^T\Sigma)^{-1}(\bmv-\bmx)
    =(\bmv-\bmx)^T\Sigma^{-2}(\bmv-\bmx)
\end{align}
where the last equality relies on the fact that $\Sigma$ is symmetric.
Combining (\ref{eqn:cov_1}), (\ref{eqn:cov_2}), and (\ref{eqn:cov_3}) give the result.
\end{proof}

\begin{proof}[Proof of \eqref{eqn:gradient_equation}]
We start by computing the partials of
\begin{equation}
    f_{\Sigma}(\bmx)
    =\frac{1}{\pi^{\frac{d}{2}}|\Sigma|}\int_{\mathbb{R}^d}f(\bmv)e^{-(\bmv-\bmx)^T\Sigma^{-2}(\bmv-\bmx)}\;d\bmv.
\end{equation}
First, the derivative of the exponent is
\begin{align}
    \frac{\partial}{\partial \bmx_i}\Big[-(\bmv-\bmx)^T\Sigma^{-2}(\bmv-\bmx)\Big]
    &=-\frac{\partial}{\partial \bmx_i}\|\Sigma^{-1}(\bmv-\bmx)\|^2\\
    &=-\frac{\partial}{\partial \bmx_i}\sum_{j=1}^{d}\big(\Sigma^{-1}(\bmv-\bmx)\big)^2_j\\
    &=-\frac{\partial}{\partial \bmx_i}\sum_{j=1}^{d}\left(\sum_{k=1}^{d}\Sigma^{-1}_{j,k}(\bmv_k-\bmx_k)\right)^2\\
    &=\sum_{j=1}^{d}2\Sigma^{-1}_{j,i}\big(\Sigma^{-1}(\bmv-\bmx)\big)_j.
\end{align}
This means
\begin{align}
    \frac{\partial}{\partial \bmx_i}f_{\Sigma}(\bmx)
    &=\frac{2}{\pi^{\frac{d}{2}}}\int_{\mathbb{R}^d}\big(\Sigma^{-1}\bmu\big)_if(\bmx+\Sigma\bmu)e^{-\|\bmu\|^2}\;d\bmu.
\end{align}
Second, we combine the partials to get the gradient
\begin{align}
    \nabla f_{\Sigma}(\bmx)
    &=\frac{2}{\pi^{\frac{d}{2}}}\int_{\mathbb{R}^d}\big(\Sigma^{-1}\bmu\big)f(\bmx+\Sigma\bmu)e^{-\|\bmu\|^2}\;d\bmu\\
    &=\frac{2}{\pi^{\frac{d}{2}}}\Sigma^{-1}\int_{\mathbb{R}^d}\bmu f(\bmx+\Sigma\bmu)e^{-\|\bmu\|^2}\;d\bmu
\end{align}
\end{proof}

\begin{proof}[Proof of Lemma~\ref{lem:smoothing_maintains_properties_nesterov} (\ref{item:lemma_c})]
The proof that $f_{\Sigma}$ is $M$-Lipschitz is the same as the isotropic case.

To see that $f_{\Sigma}$ is $M\sqrt{2d}\|\Sigma^{-1}\|$-smooth:
\begin{align}
\begin{split}
    \|\nabla f_{\Sigma}(\bmx)-\nabla f_{\Sigma}(\bmy)\|
    &=\frac{2}{\pi^{\frac{d}{2}}}\left\|
        \int_{\mathbb{R}^d}\Big(f(\bmx+\Sigma \bmu)-f(\bmy+\Sigma \bmu)\Big)\Sigma^{-1}\bmu e^{-\|\bmu\|^2}\;d\bmu
    \right\|\\
    &\leq\frac{2}{\pi^{\frac{d}{2}}}
        \int_{\mathbb{R}^d}\Big|f(\bmx+\Sigma \bmu)-f(\bmy+\Sigma \bmu)\Big|\|\Sigma^{-1}\bmu\| e^{-\|\bmu\|^2}\;d\bmu\\
    &\leq\frac{2}{\pi^{\frac{d}{2}}}\int_{\mathbb{R}^d}M\|\bmx-\bmy\|\|\Sigma^{-1}\|\|\bmu\|e^{-\|\bmu\|^2}\;d\bmu\\
    &=\frac{2}{\pi^{\frac{d}{2}}}M\|\Sigma^{-1}\|\|\bmx-\bmy\|\int_{\mathbb{R}^d}\|\bmu\|e^{-\|\bmu\|^2}\;d\bmu\\
    &=M\|\Sigma^{-1}\|\sqrt{2d}\|\bmx-\bmy\|.
\end{split}
\end{align}
\end{proof}

\begin{proof}[Proof of Lemma~\ref{lem:difference_f_and_fsigma} (\ref{item:smooth_lemma_a})]
Observe
\begin{align}
\begin{split}
    |f_{\Sigma}(\bmx)-f(\bmx)|
    &=\frac{1}{\pi^{\frac{d}{2}}}\left\|
        \int_{\mathbb{R}^d}\Big(f(\bmx+\Sigma \bmu)-f(\bmx)-\langle\nabla f(\bmx),\Sigma \bmu\rangle\Big)e^{-\|\bmu\|^2}\;d\bmu
    \right\|\\
    &\leq\frac{1}{\pi^{\frac{d}{2}}}
        \int_{\mathbb{R}^d}\Big|f(\bmx+\Sigma \bmu)-f(\bmx)-\langle\nabla f(\bmx),\Sigma \bmu\rangle\Big|e^{-\|\bmu\|^2}\;d\bmu\\
    &\leq\frac{1}{\pi^{\frac{d}{2}}}
        \int_{\mathbb{R}^d}\frac{1}{2}L\|\Sigma \bmu\|^2e^{-\|\bmu\|^2}\;d\bmu\\
    &\leq\frac{L}{2\pi^{\frac{d}{2}}}\|\Sigma\|^2
        \int_{\mathbb{R}^d}\|\bmu\|^2e^{-\|\bmu\|^2}\;d\bmu\\
    &=\frac{Ld}{4}\|\Sigma\|^2.
\end{split}
\end{align}
\end{proof}

\begin{proof}[Proof of Lemma~\ref{lem:difference_f_and_fsigma} (\ref{item:smooth_lemma_b})]
Observe
\begin{align}
\begin{split}
    &\|\nabla f_{\Sigma}(\bmx)-\nabla f(\bmx)\|\\
    &\qquad\qquad
    =\frac{2}{\pi^{\frac{d}{2}}}\left\|
        \int_{\mathbb{R}^d}\Big(\big(f(\bmx+\Sigma \bmu)-f(\bmx)\big)\Sigma^{-1}\bmu-\langle\nabla f(\bmx),\bmu\rangle I_d\bmu\Big)e^{-\|\bmu\|^2}\;d\bmu
    \right\|\\
    &\qquad\qquad
    =\frac{2}{\pi^{\frac{d}{2}}}\left\|
        \int_{\mathbb{R}^d}\Big(\big(f(\bmx+\Sigma \bmu)-f(\bmx)\big)I_d-\langle\nabla f(\bmx),\bmu\rangle\Sigma\Big)\Sigma^{-1}\bmu e^{-\|\bmu\|^2}\;d\bmu
    \right\|\\
    &\qquad\qquad
    =\frac{2}{\pi^{\frac{d}{2}}}\left\|
        \int_{\mathbb{R}^d}\Big(\big(f(\bmx+\Sigma \bmu)-f(\bmx)\big)-\langle\nabla f(\bmx),\Sigma \bmu\rangle\Big)\Sigma^{-1}\bmu e^{-\|\bmu\|^2}\;d\bmu
    \right\|\\
    &\qquad\qquad
    \leq\frac{2}{\pi^{\frac{d}{2}}}
        \int_{\mathbb{R}^d}\Big|\big(f(\bmx+\Sigma \bmu)-f(\bmx)\big)-\langle\nabla f(\bmx),\Sigma \bmu\rangle\Big|\|\Sigma^{-1}\bmu\|e^{-\|\bmu\|^2}\;d\bmu\\
    &\qquad\qquad
    \leq\frac{2}{\pi^{\frac{d}{2}}}
        \int_{\mathbb{R}^d}\frac{1}{2}L\|\Sigma \bmu\|^2\|\Sigma^{-1}\bmu\|e^{-\|\bmu\|^2}\;d\bmu\\
    &\qquad\qquad
    \leq\frac{1}{\pi^{\frac{d}{2}}}L\|\Sigma\|^2\|\Sigma^{-1}\|
        \int_{\mathbb{R}^d}\|\bmu\|^3e^{-\|\bmu\|^2}\;d\bmu\\
    &\qquad\qquad
    =L\|\Sigma\|^2\|\Sigma^{-1}\|\left(\frac{3+d}{2}\right)^{\frac{3}{2}}.
\end{split}
\end{align}
\end{proof}

\begin{proof}[Proof of Corollary~\ref{cor:nesterov_lemma_4_reversed}]
From Lemma~\ref{lem:difference_f_and_fsigma}~(\ref{item:smooth_lemma_b}),
\begin{equation}
    |\nabla f_{\Sigma}(\bmx)-\nabla f(\bmx)|^2
    \leq L^2\|\Sigma\|^4\|\Sigma^{-1}\|^2\left(\frac{3+d}{2}\right)^{3}.
\end{equation}
So,
\begin{align}
    \|\nabla f_{\Sigma}(\bmx)\|^2
    &=\|\nabla f(\bmx) + \nabla f_{\Sigma}(\bmx) - \nabla f(\bmx)\|^2\\
    &\leq 
        2\|\nabla f(\bmx)\|^2 
        + 2\|\nabla f_{\Sigma}(\bmx) - \nabla f(\bmx)\|^2\\
    &\leq
        2\|\nabla f(\bmx)\|^2
        + 2L^2\|\Sigma\|^4\|\Sigma^{-1}\|^2\left(\frac{3+d}{2}\right)^{3}\\
    &=
        2\|\nabla f(\bmx)\|^2
        +\frac{L^2\|\Sigma\|^4\|\Sigma^{-1}\|^2(3+d)^3}{4}.
\end{align}
\end{proof}

\begin{proof}[Proof of Lemma~\ref{lem:different_smoothing_values} (\ref{item:double_lemma_a})]
Before we prove the result, we justify the existence of $\Eta$. Since $\Sigma$ and $\Tau$ are symmetric, $\Sigma^2$ and $\Tau^2$ are also symmetric and so $\Sigma^2+\Tau^2$ is symmetric too. Further, for any symmetric $A$
\begin{equation}
    \bmx^TA^2\bmx=(A\bmx)^T(A\bmx)=\|A\bmx\|^2\geq 0,
\end{equation}
which means $A^2$ is positive semi-definite.
In our setting, $\Sigma^2$ and $\Tau^2$ are both positive semi-definite.
Then
\begin{equation}
    \bmx^T(\Sigma^2+\Tau^2)\bmx
    =\bmx^T\Sigma^2\bmx+\bmx^T\Tau^2\bmx
    \geq 0 + 0
    = 0
\end{equation}
shows that $\Sigma^2+\Tau^2$ is positive semi-definite as well.
As such, $\Sigma^2+\Tau^2$ has a square root, which we call $\Eta$.

Observe the following:
\begin{align}
\begin{split}
    (f_{\Sigma})_{\Tau}(\bmx)
    &=(f\star k_{\Sigma})_{\Tau}(\bmx)\\
    &=\big((f\star k_{\Sigma})\star k_{\Tau}\big)(\bmx)\\
    &=\big(f\star (k_{\Sigma}\star k_{\Tau})\big)(\bmx).
\end{split}
\end{align}
If $X\sim N(\bm{0},\Sigma^2)$ and $Y\sim N(\bm{0},\Tau^2)$ (which have densities $k_{\Sigma}$ and $k_{\Tau}$, respectively), then $X+Y\sim N(\bm{0},\Sigma^2+\Tau^2)=N(\bm{0},\Eta^2)$ (which has density $k_{\Eta}$).
We also know that the densities of a sum of random variables is generated by convolving the densities of the random variables, which means $k_{\Sigma}\star k_{\Tau}=k_{\Eta}$.
This shows
\begin{align}
\begin{split}
    (f_{\Sigma})_{\Tau}(\bmx)
    &=\big(f\star (k_{\Sigma}\star k_{\Tau})\big)(\bmx)\\
    &=(f\star k_{\Eta})(\bmx)\\
    &=f_{\Eta}(\bmx).
\end{split}
\end{align}
\end{proof}

\begin{proof}[Proof of Lemma~\ref{lem:different_smoothing_values} (\ref{item:double_lemma_b})]
First, suppose $\Sigma^2$ and $\Tau^2$ diagonalize together (i.e., they have the same eigenvectors), then
\begin{equation}
    \Sigma^2=PD_{\Sigma}P^T
    \text{ and }
    T^2=PD_{T}P^T
\end{equation}
where $D_{\Sigma}=\text{diag}(\sigma_1,...,\sigma_d)$ and $D_{T}=\text{diag}(\tau_1,...,\tau_d)$ are the diagonal matrices containing the corresponding eigenvalues (in descending order).
This means
\begin{equation}
    \Sigma^2-T^2
    =P(D_{\Sigma}-D_{T})P^T.
\end{equation}
Let
\begin{equation}
    D=\text{diag}(
        \min(\sigma_1,\tau_1),
        \min(\sigma_2,\tau_2),
        ...,
        \min(\sigma_d,\tau_d)
    ),
\end{equation}
then
\begin{align}
    \Sigma^2
    &=PD_{\Sigma}P^T
    =P(D+(D_{\Sigma}-D_T)_{+})P^T\\
    T^2
    &=PD_{T}P^T
    =P(D+(D_T-D_{\Sigma})_{+})P^T
\end{align}
where $(A)_{+}$ are all of the positive entries of $A$ with 0s everywhere else.
Then
\begin{align}
\begin{split}
    |f_{\Sigma}(\bmx)-f_{T}(\bmx)|
    &=|(f_{D})_{(D_{\Sigma}-D_T)_{+}}(\bmx)-(f_{D})_{(D_T-D_{\Sigma})_{+}}(\bmx)|\\
    &\leq|(f_{D})_{(D_{\Sigma}-D_T)_{+}}(\bmx)-f_D(\bmx)|+|f_D(\bmx)-(f_{D})_{(D_T-D_{\Sigma})_{+}}(\bmx)|\\
    &\leq\frac{Ld}{4}\Big(\|(D_{\Sigma}-D_T)_{+}\|^2+\|(D_T-D_{\Sigma})_{+}\|^2\Big)\\
    &\leq\frac{Ld}{4}\left(\left(\max_{1\leq i\leq d}(0,\sigma_i-\tau_i)\right)^2+\left(\max_{1\leq i\leq d}(0,\tau_i-\sigma_i)\right)^2\right)
\end{split}
\end{align}

Second, suppose that $\Sigma$ and $\Tau$ do not have the same eigenvectors.
Let the eigenvectors of $\Sigma^2$ and $T^2$ are $\bmu_i$ and $\bmv_i$, respectively, with corresponding eigenvalues $\sigma_i$ and $\tau_i$, respectively.
Let $\alpha=\min_{1\leq i\leq d}(\sigma_i,\tau_i)$ and $H=\text{diag}(\alpha,...,\alpha)$.
Then both $\Sigma^2-H^2$ and $T^2-H^2$ are both positive semi-definite.
We will justify this for $\Sigma^2$, but the same proof will hold for $T^2$.
First, $\bmu_i^T\Sigma^2\bmu_i=\sigma_i\|\bmu_i\|^2$.
Second, if $\bmx=\sum c_i\bmu_i$, then
\begin{align}
\begin{split}
    \bmx^T\Sigma^2\bmx
    &=\bmx^T\left(\sum c_i\sigma_i\bmu_i\right)\\
    &=\sum c_i\sigma_i\bmx^T\bmu_i\\
    &=\sum_ic_i\sigma_i\left(\sum_j c_j\bmu_j^T\bmu_i\right)\\
    &=\sum_i\sigma_ic_i^2\|\bmu_i\|^2\\
    &\geq\sum_i\alpha c_i^2\|\bmu_i\|^2\\
    &=\alpha\|\bmx\|^2
\end{split}
\end{align}
So,
\begin{align}
\begin{split}
    \bmx^T(\Sigma^2-H^2)\bmx
    &=\bmx^T\Sigma^2\bmx-\bmx^TH\bmx\\
    &=\bmx^T\Sigma^2\bmx-\alpha\|\bmx\|^2\\
    &\geq\alpha\|\bmx\|^2-\alpha\|\bmx\|^2\\
    &=0,
\end{split}
\end{align}
which shows that $\Sigma^2-H^2$ is indeed positive semi-definite.

One immediate bound is
\begin{equation}
    \|\Sigma^2-H^2\|
    \leq\|\Sigma^2\|-\min_{1\leq i\leq d}(\sigma_i,\tau_i),
\end{equation}
since if $\bmx$ is an eigenvector of $\Sigma^2-H^2$
\begin{align}
                        &\lambda \bmx=(\Sigma^2-H^2)\bmx=\Sigma^2\bmx-\alpha \bmx\\
    \Longrightarrow\quad&(\lambda +\alpha)\bmx=\Sigma^2 \bmx\leq\|\Sigma^2\|\bmx\\
    \Longrightarrow\quad&\lambda\leq\|\Sigma^2\|-\alpha
\end{align}
This makes the overall bound
\begin{equation}
\label{eqn:different_smoothing_values}
    |f_{\Sigma}(\bmx)-f_{T}(\bmx)|
    \leq\frac{Ld}{4}\left(\|\Sigma^2\|+\|T^2\|-2\min_{1\leq i\leq d}(\sigma_i,\tau_i)\right)
\end{equation}

\end{proof}

\begin{proof}[Proof of Lemma~\ref{lem:different_smoothing_values} (\ref{item:double_lemma_c})]
WLOG $\Tau^2-\Sigma^2$ is positive semi-definite.
Since $\Tau^2-\Sigma^2$ is symmetric, $\Tau^2-\Sigma^2$ has a square root.
Let $\Eta = \sqrt{\Tau^2-\Sigma^2}$.
By (\ref{item:double_lemma_a}), $(f_{\Sigma})_{\Eta}=f_{\Tau}$.
Applying the previous lemma, we have
\begin{align}
\begin{split}
    |f_{\Sigma}(\bmx)-f_{\Tau}(\bmx)|
    &=|f_{\Sigma}(\bmx)-(f_{\Sigma})_{\Eta}(\bmx)|\\
    &\leq\frac{Ld}{4}\|H\|^2\\
    &=\frac{Ld}{4}\|\sqrt{\Tau^2-\Sigma^2}\|^2\\
    &=\frac{Ld}{4}\|\Tau^2-\Sigma^2\|.
\end{split}
\end{align}
\end{proof}

\section{Proofs of \agsgd}
\label{app:agsgd_proofs}

\begin{proof}[Proof of Theorem~\ref{thm:AGSmoothGD_convex}]
(This is essentially the proof in~\cite{gsgd} with changes made on the bounds.)

Suppose $f$ is convex and $L$-smooth.
By Lemma \ref{lem:smoothing_maintains_properties_nesterov} (\ref{item:lemma_a}), $f_{\Sigma_t}$ is $L$-smooth for each $t\in\mathbb{N}$.
As such, we have
\begin{align}
\begin{split}
\label{eqn:quickproof2}
    f_{\Sigma_{t+1}}(\bmx_{t+1})
    &\leq f_{\Sigma_{t+1}}(\bmx_{t})+\langle\nabla f_{\Sigma_{t+1}}(\bmx_{t}),\bmx_{t+1}-\bmx_{t}\rangle +\frac{L}{2}\|\bmx_{t+1}-\bmx_{t}\|^2\\
    &=f_{\Sigma_{t+1}}(\bmx_{t})-\lrate\|\nabla f_{\Sigma_{t+1}}(\bmx_{t})\|^2+\frac{L\lrate^2}{2}\|\nabla f_{\Sigma_{t+1}}(\bmx_{t})\|^2\\
    &=f_{\Sigma_{t+1}}(\bmx_{t})+\lrate\left(\frac{L\lrate}{2}-1\right)\|\nabla f_{\Sigma_{t+1}}(\bmx_{t})\|^2\\
    &\leq f_{\Sigma_{t+1}}(\bmx_{t})-\frac{\lrate}{2}\|\nabla f_{\Sigma_{t+1}}(\bmx_{t})\|^2
\end{split}
\end{align}
Using Lemma~\ref{lem:different_smoothing_values} and the assumption that $|f_{\Sigma}(\bmx)-f_{T}(\bmx)|<B(\Sigma,T)$,
\begin{equation}
\label{eqn:quickproof4}
    f_{\Sigma_t}(\bmx_{t})
    \leq f_{\Sigma_t}(\bmx_{t-1})
    \leq f_{\Sigma_{t-1}}(\bmx_{t-1})+B(\Sigma_t,\Sigma_{t+1}).
\end{equation}

By Lemma~\ref{lem:smoothing_maintains_properties_nesterov} (\ref{item:lemma_b}), $f_{\Sigma_t}$ is convex for each $t\in\mathbb{N}$.
So,
\begin{equation}
\label{eqn:quickproof1}
    f_{\Sigma_t}(\bmx)-f_{\Sigma_t}(\xstar)
    \leq\langle\nabla f_{\Sigma_t}(\bmx),\bmx-\xstar\rangle.
\end{equation}
Then
\begin{align}
\begin{split}
\label{eqn:quickproof3}
    0
    &\leq
    f_{\Sigma_{t+1}}(\bmx_{t+1})-f(\xstar)\\
    &=f_{\Sigma_{t+1}}(\bmx_{t+1})-f_{\Sigma_{t+1}}(\xstar)+f_{\Sigma_{t+1}}(\xstar)-f(\xstar)\\
    &\leq
    f_{\Sigma_{t+1}}(\bmx_{t+1})-f_{\Sigma_{t+1}}(\xstar)+\frac{Ld\|\Sigma_{t+1}\|^2}{4}\\
    &\leq
    f_{\Sigma_{t+1}}(\bmx_{t})-\frac{\lrate}{2}\|\nabla f_{\Sigma_{t+1}}(\bmx_{t})\|^2-f_{\Sigma_{t+1}}(\xstar)+\frac{Ld\|\Sigma_{t+1}\|^2}{4}\\
    &\leq
    \langle\nabla f_{\Sigma_{t+1}}(\bmx_{t}),\bmx_{t}-\xstar\rangle-\frac{\lrate}{2}\|\nabla f_{\Sigma_{t+1}}(\bmx_{t})\|^2+\frac{Ld\|\Sigma_{t+1}\|^2}{4}.
\end{split}
\end{align}
Now, we have the following:
\begin{align}
\begin{split}
\label{eqn:quickproof4again}
    f_{\Sigma_{t+1}}(\bmx_{t+1})-f(\xstar)
    &\leq\langle\nabla f_{\Sigma_{t+1}}(\bmx_{t}),\bmx_{t}-\xstar\rangle-\frac{\lrate}{2}\|\nabla f_{\Sigma_{t+1}}(\bmx_{t})\|^2+\frac{Ld\|\Sigma_{t+1}\|^2}{4}\\
    &=\frac{1}{2\lrate}\left(2\lrate\langle\nabla f_{\Sigma_{t+1}}(\bmx_{t}),\bmx_{t}-\xstar\rangle-\lrate^2\|\nabla f_{\Sigma_{t+1}}(\bmx_{t})\|^2\right)+\frac{Ld\|\Sigma_{t+1}\|^2}{4}\\
    &=\frac{1}{2\lrate}\left(-\|\bmx_{t}-\lrate\nabla f_{\Sigma_{t+1}}(\bmx_{t})-\xstar\|^2+\|\bmx_{t}-\xstar\|^2\right)+\frac{Ld\|\Sigma_{t+1}\|^2}{4}\\
    &=\frac{1}{2\lrate}\left(\|\bmx_{t}-\xstar\|^2-\|\bmx_{t+1}-\xstar\|^2\right)+\frac{Ld\|\Sigma_{t+1}\|^2}{4}
\end{split}
\end{align}
Summing over the steps:
\begin{align}
\begin{split}
\label{eqn:quickproof5}
    \sum_{i=0}^{T-1}\left(f_{\Sigma_{t+1}}(\bmx_{t+1})-f(\xstar)\right)
    &\leq\frac{1}{2\lrate}\sum_{i=0}^{T-1}\left(\|\bmx_t-\xstar\|^2-\|\bmx_{t+1}-\xstar\|^2\right) + \frac{Ld}{4}\sum_{t=0}^{T-1}\|\Sigma_{t+1}\|^2\\
    &=\frac{1}{2\lrate}\|\bmx_0-\xstar\|^2 + \frac{Ld}{4}\sum_{t=0}^{T-1}\|\Sigma_{t+1}\|^2.
\end{split}
\end{align}
By (\ref{eqn:quickproof4}), for $T>t\geq 1$ we have
\begin{equation}
    f_{\Sigma_T}(\bmx_{T})
    \leq f_{\Sigma_{T-1}}(\bmx_{T-1}) + B(\Sigma_{T-1},\Sigma_{T})
    \leq\cdots
    \leq f_{\sigma_{T-t}}(\bmx_{T-t}) + \sum_{j=1}^{t}B(\Sigma_{T-j},\Sigma_{T-j+1}).
\end{equation}
This means we have
\begin{align}
\begin{split}
\label{eqn:quickproof6}
    &T\left(f_{\Sigma_{T}}(\bmx_{T})-f(\xstar)\right)\\
    &\qquad\qquad
        =(f_{\Sigma_{T}}(\bmx_{T})-f(\xstar))+\sum_{t=1}^{T-1}(f_{\Sigma_{T}}(\bmx_{T})-f(\xstar))\\
    &\qquad\qquad
        \leq(f_{\Sigma_{T}}(\bmx_{T})-f(\xstar))
    +\sum_{t=1}^{T-1}\left(f_{\Sigma_{T-t}}(x_{T-t})-f(\xstar)+\sum_{j=1}^{t}B(\Sigma_{T-j},\Sigma_{T-j+1})\right)\\
    &\qquad\qquad
        =\sum_{t=1}^{T}(f_{\Sigma_{t}}(x_t)-f(\xstar))+\sum_{t=1}^{T-1}\sum_{j=1}^{t}B(\Sigma_{T-j},\Sigma_{T-j+1})\\
    &\qquad\qquad
        =\sum_{t=1}^{T}(f_{\Sigma_{t}}(\bmx_t)-f(\xstar))+\sum_{t=1}^{T-1}tB(\Sigma_{T-t},\Sigma_{T-t+1})\\
    &\qquad\qquad
        \leq\frac{1}{2\lrate}\|\bmx_0-\xstar\|^2 + \frac{Ld}{4}\sum_{t=0}^{T-1}\|\Sigma_{T+1}\|^2+\sum_{t=1}^{T-1}tB(\Sigma_{T-t},\Sigma_{T-t+1}).
\end{split}
\end{align}
Finally, applying Lemma~\ref{lem:smoothing_maintains_properties_nesterov} (\ref{item:lemma_d}) again and dividing by $t$,
\begin{align}
\label{eqn:quickproof7}
    f(\bmx_{T})-f(\xstar)
    &\leq f_{\Sigma_{T}}(\bmx_{T})-f(\xstar)\\
    &\leq\frac{1}{2T}\|\bmx_0-\xstar\|^2+\frac{1}{T}\left(
        \frac{Ld}{4}\sum_{t=0}^{T-1}\|\Sigma_{t+1}\|^2+\sum_{t=1}^{T-1}tB(\Sigma_{T-t},\Sigma_{T-t+1})
    \right),
\end{align}
which completes the proof.
\end{proof}


\begin{proof}[Proof of Theorem~\ref{thm:AGSmoothGD_nonconvex}]
(This is essentially the proof in~\cite{gsgd} with changes made on the bounds.)

From the proof of Theorem~\ref{thm:AGSmoothGD_convex}, since $f$ is $L$-smooth (regardless of convexity)
\begin{equation}
    f_{\Sigma_{t+1}}(\bmx_{t+1})
    \leq f_{\Sigma_{t+1}}(\bmx_{t})-\frac{t}{2}\|\nabla f_{\Sigma_{t+1}}(\bmx_{t})\|^2.
\end{equation}
As expected, this means $f_{\Sigma_{t+1}}(\bmx_{t})\geq f_{\Sigma_{t+1}}(\bmx_{t+1})$.

For $t\geq 1$,
\begin{align}
\begin{split}
\label{eqn:ncgdwss1}
    \|\nabla f_{\Sigma_{t+1}}(\bmx_{t})\|^2
    &\leq\frac{2}{t}\Big(f_{\Sigma_{t+1}}(\bmx_{t})-f_{\Sigma_{t+1}}(\bmx_{t+1})\Big)\\
    &\leq\frac{2}{t}\Big(f_{\Sigma_{t}}(\bmx_{t})-f_{\Sigma_{t+1}}(\bmx_{t+1}) + B(\Sigma_{t+1},\Sigma_{t})\Big)\\
    &=\frac{2}{t}\Big(f_{\Sigma_{t}}(\bmx_{t})-f_{\Sigma_{t+1}}(\bmx_{t+1})\Big) + \frac{2}{t}B(\Sigma_{t+1},\Sigma_{t}).
\end{split}
\end{align}

This means that
\begin{align}
\begin{split}
    &T\min_{t=1,...,T}\|\nabla f(\bmx_t)\|^2\\
    &\leq\sum_{t=1}^{T}\|\nabla f(\bmx_t)\|^2\\
    &\leq 2\sum_{t=1}^{T}\left(\|\nabla f_{\Sigma_t}(\bmx_t)\|^2+L^2\|\Sigma_t\|^4\|\Sigma_t^{-1}\|^2\left(\frac{6+d}{2}\right)^3\right)\\
    &\leq 
        \frac{4}{T}\sum_{t=1}^{T}\big(f_{\Sigma_{t}}(\bmx_{t})-f_{\Sigma_{t+1}}(\bmx_{t+1})\big)
        +\frac{4}{T}\sum_{t=1}^{T}B(\Sigma_{t+1},\Sigma_{t})
        +L^2\left(\frac{6+d}{2}\right)^3\sum_{t=1}^{T}\|\Sigma_t\|^4\|\Sigma_t^{-1}\|^2\\
    &=
        \frac{4}{T}\big(f_{\Sigma_{1}}(\bmx_{1})-f_{\Sigma_{T+1}}(\bmx_{T+1})\big)
        +\frac{4}{T}\sum_{t=1}^{T}B(\Sigma_{t+1},\Sigma_{t})
        +L^2\left(\frac{6+d}{2}\right)^3\sum_{t=1}^{T}\|\Sigma_t\|^4\|\Sigma_t^{-1}\|^2.
\end{split}
\end{align}
Dividing both sides by $T$ gives the result.
\end{proof}

\section{Proofs of \agssgd}
\label{app:agssgd_proofs}

\begin{proof}[Proof of Theorem~\ref{thm:GSSGD}]
Note that this is a modification of the proof of convergence for GSmoothSGD from~\cite{gssgd}.
First, since $\mathbb{E}\big[\|\nabla f_{k}\|^2\big]\leq\lambda$, from Corollary~\ref{cor:nesterov_lemma_4_reversed}, we have that
\begin{equation}
    \mathbb{E}\big[\|\nabla f_{k,\Sigma}\|^2\big]
    \leq 2\lambda^2+\frac{1}{4}L^2\|\Sigma\|^4\|\Sigma^{-1}\|^2(3+d)^3.
\end{equation}
For $\Sigma_t$, denote this bound as $\lambda_t$, that is
\begin{equation}
    \mathbb{E}\big[\|\nabla f_{k,\Sigma_t}\|^2\big]\leq\lambda_t.
\end{equation}
Second, since $f$ is $L$-smooth, so is $f_{\Sigma_t}$ for all $t$.

Now, we repeat the analysis done in GSmoothSGD.
Since $f_{\Sigma_{t+1}}$ is $L$-smooth,
\begin{equation}
    f_{\Sigma_{t+1}}(\bmx_{t+1})
    =
        f_{\Sigma_{t+1}}(\bmx_t)
        -\eta_t\langle\nabla f_{\Sigma_{t+1}}(\bmx_t),\nabla f_{k_t,\Sigma_{t+1}}(\bmx_t)\rangle
        +\frac{L\eta_t^2}{2}\|\nabla f_{k_t,\Sigma_{t+1}}(\bmx_t)\|^2
\end{equation}
Applying the expectation to this bound, we have
\begin{align}
\begin{split}
    \mathbb{E}\big[f_{\Sigma_{t+1}}(\bmx_{t+1})\big]
    &\leq
        \mathbb{E}\big[f_{\Sigma_{t+1}}(\bmx_{t})\big]
        -\eta_t \mathbb{E}\big[\|\nabla f_{\Sigma_{t+1}}(\bmx_t)\|^2\big]
        +\frac{L\eta_t^2}{2}\lambda_{t+1}.
\end{split}
\end{align}
Rearranging we have
\begin{align}
\begin{split}
    \eta_t \mathbb{E}\big[\|\nabla f_{\Sigma_{t+1}}(\bmx_t)\|^2\big]
    &\leq
        \mathbb{E}\big[f_{\Sigma_{t+1}}(\bmx_{t})-f_{\Sigma_{t+1}}(\bmx_{t+1})\big]
        +\frac{L\eta_t^2}{2}\lambda_{t+1}.
\end{split}
\end{align}
As in the isotropic case, we need to switch from $\Sigma_{t+1}$ to $\Sigma_t$ in order to have a telescoping sum.
Here, in the anisotropic case, our bound is more complicated because it can improve based on the relationship between $\Sigma_{t+1}$ and $\Sigma_t$.
We denote by $B(\Sigma_{t+1},\Sigma_t)$ the bound that satisfies
\begin{equation}
    |f_{\Sigma_{t+1}}(\bmx)-f_{\Sigma_t}(\bmx)|
    \leq B(\Sigma_{t+1},\Sigma_t).
\end{equation}
So, we have
\begin{equation}
    \eta_t \mathbb{E}\big[\|\nabla f_{\Sigma_{t+1}}(\bmx_t)\|^2\big]
    \leq
        \mathbb{E}\big[f_{\Sigma_{t}}(\bmx_{t})-f_{\Sigma_{t+1}}(\bmx_{t+1})\big]
        +\frac{L\eta_t^2}{2}\lambda_{t+1}
        +B(\Sigma_{t+1},\Sigma_t).
\end{equation}
Continuing the analysis from GSmoothSGD, after summing over the steps
\begin{align}
\begin{split}
    \sum_{t=1}^{T}\eta_t\mathbb{E}\big[\|\nabla f_{\Sigma_{t}}(\bmx_{t-1})\|^2\big]
    &\leq
        \sum_{t=1}^{T}\mathbb{E}\big[f_{\Sigma_{t-1}}(\bmx_{t-1})-f_{\Sigma_{t}}(\bmx_{t})\big]
        +\frac{L}{2}\sum_{t=1}^{T}\eta_t^2\lambda_{t}
        +\sum_{t=1}^{T}B(\Sigma_{t},\Sigma_{t-1})\\
    &=
        E(f_{\Sigma_{0}}(\bmx_0)-f_{\Sigma_{T}}(\bmx_{T}))
        +\frac{L}{2}\sum_{t=1}^{T}\eta_t^2\lambda_{t}
        +\sum_{t=1}^{T}B(\Sigma_{t},\Sigma_{t-1})\\
    &\leq
        f(\bmx_0)-f_*
        +\frac{L}{2}\sum_{t=1}^{T}\eta_t^2\lambda_{t}
        +\sum_{t=1}^{T}B(\Sigma_{t},\Sigma_{t-1})
\end{split}
\end{align}
where we use $\Sigma_0=0$ for notational convenience.
Rearranging, we have
\begin{align}
\begin{split}
    \sum_{t=1}^{T}\mathbb{E}\big[\|\nabla f_{\Sigma_{t}}(\bmx_{t-1})\|^2\big]
    &\leq
        \frac{f(\bmx_0)-f_*}{\sum_{t=1}^{T}\eta_t}
        +\frac{L}{2}\frac{\sum_{t=1}^{T}\eta_t^2\lambda_{t}}{\sum_{t=1}^{T}\eta_t}
        +\frac{\sum_{t=1}^{T}B(\Sigma_{t},\Sigma_{t-1})}{\sum_{t=1}^{T}\eta_t}
\end{split}
\end{align}

After replacing $\lambda_t$, the minimum satisfies
\begin{multline}
    \min_{1\leq t\leq T}\mathbb{E}\big[\|\nabla f_{\Sigma_{t+1}}(\bmx_{t-1})\|^2\big]
    \leq
        \frac{f(\bmx_0)-f_*}{\sum_{t=1}^{T}\eta_t}
        +\frac{L\lambda^2\sum_{t=1}^{T}\eta_t^2}{\sum_{t=1}^{T}\eta_t}\\
        +\frac{L^3(3+d)^3}{8}\frac{\sum_{t=1}^{T}\|\Sigma_t\|^4\|\Sigma_t^{-1}\|^2\eta_t^2}{\sum_{t=1}^{T}\eta_t}
        +\frac{\sum_{t=1}^{T}B(\Sigma_{t},\Sigma_{t-1})}{\sum_{t=1}^{T}\eta_t}.
\end{multline}
\end{proof}

\section{Proofs of \agsadam}
\label{app:agsadam_proofs}

The following proof is the same as the proof in~\cite{gssgd}, where the only change is replacing the bounds from
\begin{equation}
\label{eqn:agsadam_1d_grad_bound}
    \|\nabla f_{\sigma_{t+1}}(\bmx_t)-\nabla f_{\sigma_{t}}(\bmx_t)\|
    \leq L\left(\frac{3+d}{2}\right)^{\nicefrac{3}{2}}\sqrt{|\sigma_{t+1}^2-\sigma_{t}^2|}
\end{equation}
with the analogous bound given by \eqref{eqn:agsadam_grad_bound}.

\begin{lem}[Lemma C.1~\cite{gsgd}]
\label{lem:adam_lemma19}
Let $(\bmx_t)_{t\geq 1}$, $(\bmm_t)_{t\geq 1}$, and $(\bmv_t)_{t\geq 1}$ be the sequences generated by \agsadam.
Let $f$ be $L$-smooth and $f^*$ denote the minimum of $f$.
Assume $\mathbb{E}\big[f_k(\bmx)\big]=f(\bmx)$ and $\mathbb{E}\big[\nabla f_k(\bmx)\big]=\nabla f(\bmx)$.
Suppose that $\mathbb{E}\big[\|\nabla f_k\|^2\big]\leq\lambda$ for any $k\in[K]$.
Then for all $t\geq 1$, we have
\begin{equation}
    \mathbb{E}\left[\left\langle
        \nabla f_{\Sigma_{t+1}}(\bmx_t),
        \frac{\bmm_{t+1}}{\sqrt{\bmv_{t+1}+\epsilon}}
    \right\rangle\right]
    \geq\sum_{i=1}^{t+1}\prod_{j=i+1}^{t+1}\beta_jD_i
    -\frac{M}{\sqrt{\epsilon}}\sum_{i=1}^{t+1}\prod_{j=i}^{t+1}\beta_j\widetilde{B}(\Sigma_{i},\Sigma_{i-1}),
\end{equation}
where
\begin{equation}
    D_i
    =-L\beta_{i+1}\eta_{i}E\left(\left\|\frac{\bmm_{i}}{\sqrt{\bmv_{i}+\epsilon}}\right\|^2\right)
    +\frac{1-\beta_{i+1}}{\sqrt{M^2+\epsilon}}\mathbb{E}\big[\|\nabla f_{\Sigma_{i+1}}(\bmx_i)\|^2\big]
    -\frac{\sqrt{d}M^4}{\epsilon^{\nicefrac{3}{2}}}(1-\theta_{i})
\end{equation}
and $\mathbb{E}\big[\|\nabla f_k\|^2\big]\leq M$ which exists by Lemma~\ref{lem:difference_f_and_fsigma}.
\end{lem}

\begin{proof}
Let
\begin{align}
\label{eqn:lemma_19_0}
    \Theta_{t+1}
    &=\mathbb{E}\left[\left\langle
        \nabla f_{\Sigma_{t+1}}(\bmx_t),
        \frac{\bmm_{t+1}}{\sqrt{\bmv_{t+1}+\epsilon}}
    \right\rangle\right]\\
    &=\underbrace{
        \mathbb{E}\left[\left\langle
            \nabla f_{\Sigma_{t+1}}(\bmx_t),
            \frac{\bmm_{t+1}}{\sqrt{\bmv_{t}+\epsilon}}
        \right\rangle\right]
    }_{I}
    +\underbrace{
        \mathbb{E}\left[\left\langle
            \nabla f_{\Sigma_{t+1}}(\bmx_t),
            \frac{\bmm_{t+1}}{\sqrt{\bmv_{t+1}+\epsilon}}-\frac{\bmm_{t+1}}{\sqrt{\bmv_{t}+\epsilon}}
        \right\rangle\right]
    }_{II}.
\end{align}
Focusing on $I$, we have
\begin{align}
\label{eqn:lemma_19_1}
    &\mathbb{E}\left[\left.
        \left\langle
            \nabla f_{\Sigma_{t+1}}(\bmx_t),
            \frac{\bmm_{t+1}}{\sqrt{\bmv_{t+1}+\epsilon}}
        \right\rangle
    \right|\mathcal{F}_t\right]\\
    &=\mathbb{E}\left[\left.
        \left\langle
            \nabla f_{\Sigma_{t+1}}(\bmx_t),
            \frac{\beta_{t+1}\bmm_{t}+(1-\beta_{t+1})\nabla f_{k_t,\Sigma_{t+1}}(\bmx_t)}{\sqrt{\bmv_{t+1}+\epsilon}}
        \right\rangle
    \right|\mathcal{F}_t\right]\\
    &=\beta_{t+1}\left\langle
        \nabla f_{\Sigma_{t+1}}(\bmx_t),
        \frac{\bmm_{t}}{\sqrt{\bmv_{t+1}+\epsilon}}
    \right\rangle
    +(1-\beta_{t+1})\left\langle
        \nabla f_{\Sigma_{t+1}}(\bmx_t),
        \frac{\nabla f_{\Sigma_{t+1}}(\bmx_t)}{\sqrt{\bmv_{t+1}+\epsilon}}
    \right\rangle\\
    &=\beta_{t+1}\left\langle
        \nabla f_{\Sigma_{t+1}}(\bmx_t),
        \frac{\bmm_{t}}{\sqrt{\bmv_{t+1}+\epsilon}}
    \right\rangle
    +(1-\beta_{t+1})\left\|
        \frac{(\nabla f_{\Sigma_{t+1}}(\bmx_t))^2}{\sqrt{\bmv_{t+1}+\epsilon}}
    \right\|_1\\
    &=\beta_{t+1}\left\langle
        \nabla f_{\Sigma_{t}}(\bmx_{t-1}),
        \frac{\bmm_{t}}{\sqrt{\bmv_{t+1}+\epsilon}}
    \right\rangle
    +(1-\beta_{t+1})\left\|
        \frac{(\nabla f_{\Sigma_{t+1}}(\bmx_t))^2}{\sqrt{\bmv_{t+1}+\epsilon}}
    \right\|_1\\
    &\qquad\qquad\qquad-\beta_{t+1}\left\langle
        \nabla f_{\Sigma_{t+1}}(\bmx_{t})-\nabla f_{\Sigma_{t}}(\bmx_{t-1}),
        \frac{\bmm_{t}}{\sqrt{\bmv_{t+1}+\epsilon}}
    \right\rangle,
\end{align}
where $(\nabla f_{\Sigma_{t+1}}(\bmx_t))^2$ is done coordinate-wise.
Focusing on the last term of the previous equation, we have
\begin{multline}
    -\beta_{t+1}\left\langle
        \nabla f_{\Sigma_{t+1}}(\bmx_{t})-\nabla f_{\Sigma_{t}}(\bmx_{t-1}),
        \frac{\bmm_{t}}{\sqrt{\bmv_{t+1}+\epsilon}}
    \right\rangle\\
    \geq-\beta_{t+1}
    \|\nabla f_{\Sigma_{t+1}}(\bmx_{t})-\nabla f_{\Sigma_{t}}(\bmx_{t-1})\|
    \left\|\frac{\bmm_{t}}{\sqrt{\bmv_{t+1}+\epsilon}}\right\|.
\end{multline}
Note that
\begin{align}
    \|\nabla f_{\Sigma_{t+1}}(\bmx_{t})-\nabla f_{\Sigma_{t}}(\bmx_{t-1})\|
    &\leq
        \|\nabla f_{\Sigma_{t}}(\bmx_{t})-\nabla f_{\Sigma_{t}}(\bmx_{t-1})\|
        +\|\nabla f_{\Sigma_{t+1}}(\bmx_{t})-\nabla f_{\Sigma_{t}}(\bmx_{t})\|\\
    &\leq
        L\|\bmx_t-\bmx_{t-1}\|
        +\widetilde{B}(\Sigma_{t+1},\Sigma_t)
\end{align}
using the fact that $f_{\Sigma_t}$ is L-smooth.
So,
\begin{align}
    &-\beta_{t+1}\left\langle
        \nabla f_{\Sigma_{t+1}}(\bmx_{t})-\nabla f_{\Sigma_{t}}(\bmx_{t-1}),
        \frac{\bmm_{t}}{\sqrt{\bmv_{t+1}+\epsilon}}
    \right\rangle\\
    &\geq-\beta_{t+1}
    L\|\bmx_t-\bmx_{t-1}\|
    \left\|\frac{\bmm_{t}}{\sqrt{\bmv_{t+1}+\epsilon}}\right\|
    -\beta_{t+1}
    \widetilde{B}(\Sigma_{t+1},\Sigma_t)
    \left\|\frac{\bmm_{t}}{\sqrt{\bmv_{t+1}+\epsilon}}\right\|\\
    &=-\beta_{t+1}
    L\eta_{t}
    \left\|\frac{\bmm_{t}}{\sqrt{\bmv_{t+1}+\epsilon}}\right\|^2
    -\beta_{t+1}
    \widetilde{B}(\Sigma_{t+1},\Sigma_t)
    \left\|\frac{\bmm_{t}}{\sqrt{\bmv_{t+1}+\epsilon}}\right\|.
\end{align}
Combining this with where we were in (\ref{eqn:lemma_19_1}), we have
\begin{align}
    &\mathbb{E}\left[\left.
        \left\langle
            \nabla f_{\Sigma_{t+1}}(\bmx_t),
            \frac{\bmm_{t+1}}{\sqrt{\bmv_{t+1}+\epsilon}}
        \right\rangle
    \right|\mathcal{F}_t\right]\\
    &\geq\beta_{t+1}\left\langle
        \nabla f_{\Sigma_{t}}(\bmx_{t-1}),
        \frac{\bmm_{t}}{\sqrt{\bmv_{t+1}+\epsilon}}
    \right\rangle
    +(1-\beta_{t+1})\left\|
        \frac{(\nabla f_{\Sigma_{t+1}}(\bmx_t))^2}{\sqrt{\bmv_{t+1}+\epsilon}}
    \right\|_1\\
    &\qquad\qquad\qquad-\beta_{t+1}
        L\eta_{t}
        \left\|\frac{\bmm_{t}}{\sqrt{\bmv_{t+1}+\epsilon}}\right\|^2
        -\beta_{t+1}
        \widetilde{B}(\Sigma_{t+1},\Sigma_t)
        \left\|\frac{\bmm_{t}}{\sqrt{\bmv_{t+1}+\epsilon}}\right\|.
\end{align}
So,
\begin{align}
    I
    &=\mathbb{E}\left[\mathbb{E}\left[\left.
        \left\langle
            \nabla f_{\Sigma_{t+1}}(\bmx_t),
            \frac{\bmm_{t+1}}{\sqrt{\bmv_{t+1}+\epsilon}}
        \right\rangle
    \right|\mathcal{F}_t\right]\right]\\
    &\geq
        -\beta_{t+1}\mathbb{E}\left[\left\langle
            \nabla f_{\Sigma_{t}}(\bmx_{t-1}),
            \frac{\bmm_{t}}{\sqrt{\bmv_{t+1}+\epsilon}}
        \right\rangle\right]
        +(1-\beta_{t+1})\mathbb{E}\left[\left\|
            \frac{(\nabla f_{\Sigma_{t+1}}(\bmx_t))^2}{\sqrt{\bmv_{t+1}+\epsilon}}
        \right\|_1\right]\\
    &\qquad\qquad\qquad
        -\beta_{t+1} L\eta_{t}\mathbb{E}\left[\left\|\frac{\bmm_{t}}{\sqrt{\bmv_{t+1}+\epsilon}}\right\|^2\right]\\
    &\qquad\qquad\qquad
        -\beta_{t+1}\widetilde{B}(\Sigma_{t+1},\Sigma_t)
        \mathbb{E}\left[\left\|\frac{\bmm_{t}}{\sqrt{\bmv_{t+1}+\epsilon}}\right\|\right]\\
    &\geq
        \beta_{t+1}\Theta_{t}
        +\frac{1-\beta_{t+1}}{\sqrt{M^2+\epsilon}}\mathbb{E}\left[\left\|
            \nabla f_{\Sigma_{t+1}}(\bmx_t)
        \right\|^2\right]
        -\beta_{t+1} L\eta_{t}\mathbb{E}\left[\left\|\frac{\bmm_{t}}{\sqrt{\bmv_{t+1}+\epsilon}}\right\|^2\right]\\
    &\qquad\qquad\qquad
        -\frac{\beta_{t+1}M}{\sqrt{\epsilon}}\widetilde{B}(\Sigma_{t+1},\Sigma_t)
\end{align}
where the last inequality used Lemmas 16 and 17 from \cite{he2023convergence}.

Now focusing on $II$ in (\ref{eqn:lemma_19_0}),
\begin{align}
    II
    &=\mathbb{E}\left[\left\langle
        \nabla f_{\Sigma_{t+1}}(\bmx_t),
        \frac{\bmm_{t+1}}{\sqrt{\bmv_{t+1}+\epsilon}}-\frac{\bmm_{t+1}}{\sqrt{\bmv_{t}+\epsilon}}
    \right\rangle\right]\\
    &=-\mathbb{E}\left[\left\langle
        \nabla f_{\Sigma_{t+1}}(\bmx_t),
        \frac{\bmm_{t+1}}{\sqrt{\bmv_{t}+\epsilon}}-\frac{\bmm_{t+1}}{\sqrt{\bmv_{t+1}+\epsilon}}
    \right\rangle\right]\\
    &\geq \mathbb{E}\left[
        \|\nabla f_{\Sigma_{t+1}}(\bmx_t)\|
        \left\|\frac{\bmm_{t+1}}{\sqrt{\bmv_{t}+\epsilon}}-\frac{\bmm_{t+1}}{\sqrt{\bmv_{t+1}+\epsilon}}\right\|
    \right]\\
    &\geq-\frac{\sqrt{d}M^4}{\epsilon^{\frac{3}{2}}}(1-\theta_t)
\end{align}
by Lemma 18 from \cite{he2023convergence}.

Therefore, using the definitions of $D_t$ and $\widetilde{D}_t$ from the statement of the lemma,
\begin{align}
    \Theta_{t+1}
    &\geq\beta_{t+1}\Theta_{t}
        +\frac{1-\beta_{t+1}}{\sqrt{M^2+\epsilon}}\mathbb{E}\left[\left\|
            \nabla f_{\Sigma_{t+1}}(\bmx_t)
        \right\|^2\right]
        -\beta_{t+1} L\eta_{t}\mathbb{E}\left[\left\|\frac{\bmm_{t}}{\sqrt{\bmv_{t+1}+\epsilon}}\right\|^2\right]\\
    &\qquad\qquad\qquad
        -\frac{\sqrt{d}M^4}{\epsilon^{\frac{3}{2}}}(1-\theta_{t+1})
        -\frac{\beta_{t+1}M}{\sqrt{\epsilon}}\widetilde{B}(\Sigma_{t+1},\Sigma_t)\\
    &=\beta_{t+1}\Theta_{t}
        +D_t
        -\frac{M}{\sqrt{\epsilon}}\beta_{t+1}\widetilde{B}(\Sigma_{t+1},\Sigma_t)
\end{align}
Recursively applying this inequality, we see that
\begin{align}
    \Theta_{t+1}
    &\geq\prod_{i=1}^{t+1}\beta_i\Theta_0
        +\sum_{i=1}^{t+1}\prod_{j=i+1}^{t+1}\beta_jD_i
        -\frac{M}{\sqrt{\epsilon}}\sum_{i=1}^{t+1}\prod_{j=i}^{t+1}\beta_j\widetilde{B}(\Sigma_{i},\Sigma_{i-1})\\
    &=
        \sum_{i=1}^{t+1}\prod_{j=i+1}^{t+1}\beta_jD_i
        -\frac{M}{\sqrt{\epsilon}}\sum_{i=1}^{t+1}\prod_{j=i}^{t+1}\beta_j\widetilde{B}(\Sigma_{i},\Sigma_{i-1})
\end{align}
since $\Theta_0=0$ and we used $\prod_{j=t+2}^{t+1}\beta_{t}=1$ for notational convenience.
\end{proof}

The next theorem is the \agsadam analogue of Theorem~\ref{thm:GSSGD} and the proof is the same as that of Theorem C.1 in~\cite{gssgd}, which itself is a modification of Theorem 4 from~\cite{he2023convergence}.
There are three changes to the proof from~\cite{gssgd}.
First, we use the same change in bounds as in the previous lemma.
Second, we use Corollary~\ref{cor:nesterov_lemma_4_reversed} to bound $\|\nabla f_{\Sigma}(\bmx)-\nabla f(\bmx)\|$.
Third, we use $B(\Sigma_{t+1},\Sigma_t)$ to bound $|f_{\Sigma_{t+1}}(\bmx_t)-f_{\Sigma_{t}}(\bmx_t)|$.

\begin{thm}[Theorem C.1 from~\cite{gssgd}]
\label{thm:adam_theorem4}
Let $f$ be $L$-smooth and $f^*$ denote the minimum of $f$.
Assume $\mathbb{E}\big[f_k(\bmx)\big]=f(\bmx)$ and $\mathbb{E}\big[\nabla f_k(\bmx)\big]=\nabla f(\bmx)$.
Suppose that $\mathbb{E}\big[\|\nabla f_k\|^2\big]\leq\lambda$ for any $k\in[K]$ and $(\alpha_t)_{t\geq 1}$ is a non-increasing real sequence.
Let $M\geq\mathbb{E}\big[\|\nabla f_k\|^2\big]$ which exists by Lemma~\ref{lem:difference_f_and_fsigma}.
Suppose for some $\widetilde{M}>0$ that
\begin{equation}
    \widetilde{B}(\Sigma_{t+1},\Sigma_t)\leq \widetilde{M}\eta_t
    \text{ and }
    \sum_{t=1}^{\infty}B(\Sigma_t,\Sigma_{t-1})<\infty.
\end{equation}
Let $(\bmx_t)_{t\geq 1}$ be generated by \agsadam with $\eta_t=\Theta(\alpha_t)$ (i.e., there exist $C_0,\widetilde{C_0}>0$ such that $C_0\alpha_t\leq\eta_t\leq\widetilde{C_0}\alpha_t$). Then for $T\geq 1$, we have
\begin{align}
    \min_{1\leq t\leq T}\mathbb{E}\big[\|\nabla f_{\Sigma_t}(\bmx_{t-1})\|^2\big]
    &\leq\frac{C_1+C_2\sum_{t=1}^{T}\eta_t(1-\theta_t)+C_3\sum_{t=1}^{T}\eta_t^2}{\sum_{t=1}^{T}\eta_t}.
\end{align}
where
\begin{align}
    C_1&=\frac{\sqrt{M^2+\epsilon}}{1-\beta}\left(f(\bmx_0)-f^{\star}-\sum_{t=1}^{T}B(\Sigma_t,\Sigma_{t-1})\right)\\
    C_2&=\frac{\sqrt{d}M^4\widetilde{C_0}\sqrt{M^2+\epsilon}}{\epsilon^{\nicefrac{3}{2}}C_0(1-\beta)^2}\\
    C_3&=\frac{M^2LC_0^2(2+\widetilde{M})\sqrt{M^2+\epsilon}}{\epsilon C_0^2(1-\beta)^2}
\end{align}
\end{thm}

\begin{proof}
Repeating what was done at the beginning of the \agssgd proof, we have
\begin{align}
    f_{\Sigma_{t+1}}(\bmx_{t+1})
    &\leq f_{\Sigma_{t+1}}(\bmx_t)+\langle\nabla f_{\Sigma_{t+1}}(\bmx_t),\bmx_{t+1}-\bmx_t\rangle+\frac{L}{2}\|\bmx_{t+1}-\bmx_t\|^2\\
    &=
        f_{\Sigma_{t+1}}(\bmx_t)
        -\eta_{t+1}\left\langle\nabla f_{\Sigma_{t+1}}(\bmx_t),\frac{\bmm_{t+1}}{\sqrt{\bmv_{t+1}+\epsilon}}\right\rangle
        +\frac{L\eta_{t+1}^2}{2}\left\|\frac{\bmm_{t+1}}{\sqrt{\bmv_{t+1}+\epsilon}}\right\|^2.
\end{align}
This means
\begin{align}
    &\mathbb{E}\big[f_{\Sigma_{t+1}}(\bmx_{t+1})\big]\\
    &\leq
        \mathbb{E}\big[f_{\Sigma_{t+1}}(\bmx_t)\big]
        -\eta_{t+1}\mathbb{E}\left[\left\langle\nabla f_{\Sigma_{t+1}}(\bmx_t),\frac{\bmm_{t+1}}{\sqrt{\bmv_{t+1}+\epsilon}}\right\rangle\right]
        +\frac{L\eta_{t+1}^2}{2}\mathbb{E}\left[\left\|\frac{\bmm_{t+1}}{\sqrt{\bmv_{t+1}+\epsilon}}\right\|^2\right]\\
    &\leq
    \mathbb{E}\big[f_{\Sigma_{t+1}}(\bmx_t)\big]
        -\eta_{t+1}\sum_{i=1}^{t+1}\prod_{j=i+1}^{t+1}\beta_j D_i
        +\eta_{t+1}\frac{M}{\sqrt{\epsilon}}\sum_{i=1}^{t+1}\prod_{j=i}^{t+1}\beta_j\widetilde{B}(\Sigma_{i},\Sigma_{i-1})
        +\frac{L\eta_{t+1}^2M^2}{2\epsilon}\\
    &\leq
    \mathbb{E}\big[f_{\Sigma_{t+1}}(\bmx_t)\big]
        -\eta_{t+1}\sum_{i=1}^{t+1}\prod_{j=i+1}^{t+1}\beta_j D_i
        +\eta_{t+1}\frac{M\widetilde{M}}{\sqrt{\epsilon}}\sum_{i=1}^{t+1}\prod_{j=i}^{t+1}\beta_j\eta_i
        +\frac{L\eta_{t+1}^2M^2}{2\epsilon},
\end{align}
where the second inequality comes from Lemma~\ref{lem:adam_lemma19} and \cite{he2023convergence} Lemma 16 and the last inequality is by assumption.
Repeating the analysis that was done in \cite{he2023convergence} on $D_i$ (paragraph after equation (22)) and using $\beta_j<\beta$, the smoothed version of their equation (23) is
\begin{multline}
    \mathbb{E}\big[f_{\Sigma_{t+1}}(\bmx_{t+1})\big]\\
    \leq
        \mathbb{E}\big[f_{\Sigma_{t+1}}(\bmx_t)\big]
        -\frac{(1-\beta)\eta_{t+1}}{\sqrt{M^2+\epsilon}}\mathbb{E}\big[\|\nabla f_{\Sigma_{t+1}}(\bmx_t)\|^2\big]
        +\frac{\beta L\eta_{t+1}M^2}{\epsilon}\sum_{i=1}^{t+1}\beta^{t+1-i}\eta_{i}\\
        +\frac{\sqrt{d}M^4\eta_{t+1}}{\epsilon^{\nicefrac{3}{2}}}\sum_{i=1}^{t+1}\beta^{t+1-i}(1-\theta_{i})
        +\eta_{t+1}\frac{\beta M\widetilde{M}}{\sqrt{\epsilon}}\sum_{i=1}^{t+1}\beta^{t+1-i}\eta_i
        +\frac{L\eta_{t+1}^2M^2}{2\epsilon}.
\end{multline}
Rearranging and summing,
\begin{multline}
    \frac{1-\beta}{\sqrt{M^2+\epsilon}}\sum_{t=1}^{T}\eta_t\mathbb{E}\big[\|\nabla f_{\Sigma_t}(\bmx_{t-1})\|^2\big]\\
    \leq
        \sum_{t=1}^{T}\Big(\mathbb{E}\big[f_{\Sigma_t}(\bmx_{t-1})-\mathbb{E}\big[f_{\Sigma_t}(\bmx_{t})\big]\Big)
        +\frac{M^2L}{\epsilon}\sum_{t=1}^{T}\left(\eta_t\sum_{i=1}^{t}\beta^{t-i}\eta_i\right)\\
        +\frac{M^2L}{2\epsilon}\sum_{t=1}^{T}\eta_t^2
        +\frac{\sqrt{d}M^4}{\epsilon^{\nicefrac{3}{2}}}\sum_{t=1}^{T}\left(\eta_t\sum_{i=1}^{t}\beta^{t-i}(1-\theta_i)\right)\\
        +\frac{M\widetilde{M}\beta}{\sqrt{\epsilon}}\sum_{t=1}^{T}\left(\eta_t\sum_{i=1}^{t}\beta^{t-i}\eta_i\right).
\end{multline}
From equations (25) and (26) of \cite{he2023convergence}, we have that
\begin{align}
    \sum_{t=1}^{T}\eta_t\sum_{i=1}^{t}\beta^{t-i}\eta_i
        &\leq\frac{\widetilde{C_0}^2}{C_0^2(1-\beta)}\sum_{t=1}^{T}\eta_t^2\\
    \sum_{t=1}^{T}\eta_t\sum_{i=1}^{t}\beta^{t-i}(1-\theta_i)
        &\leq\frac{\widetilde{C_0}}{C_0(1-\beta)}\sum_{t=1}^{T}\eta_t(1-\theta_t).
\end{align}
So,
\begin{align}
    &\frac{1-\beta}{\sqrt{M^2+\epsilon}}\sum_{t=1}^{T}\eta_t\mathbb{E}\big[\|\nabla f_{\Sigma_t}(\bmx_{t-1})\|^2\big]\\
    &\leq
        \mathbb{E}\big[f_{\Sigma_0}(\bmx_0)\big]-\mathbb{E}\big[f_{\Sigma_T}(\bmx_T)\big] + \sum_{t=1}^{T}B(\Sigma_t,\Sigma_{t-1})
        +\frac{M^2L\widetilde{C_0}^2}{\epsilon C_0^2(1-\beta)}\sum_{t=1}^{T}\eta_t^2
        +\frac{M^2L}{2\epsilon}\sum_{t=1}^{T}\eta_t^2\\
    &\qquad\qquad
        +\frac{\sqrt{d}M^4\widetilde{C_0}}{\epsilon^{\nicefrac{3}{2}}C_0(1-\beta)}\sum_{t=1}^{T}\eta_t(1-\theta_t)
        +\frac{M\widetilde{M}\beta\widetilde{C_0}^2}{\sqrt{\epsilon}C_0^2(1-\beta)}\sum_{t=1}^{T}\eta_t^2.
\end{align}
Then
\begin{align}
    &\sum_{t=1}^{T}\eta_t\mathbb{E}\big[\|\nabla f_{\Sigma_t}(\bmx_{t-1})\|^2\big]\\
    &\leq
        \frac{\sqrt{M^2+\epsilon}}{1-\beta}\big(f(\bmx_0)-f^{\star}\big)
        +\frac{\sqrt{M^2+\epsilon}}{1-\beta}\sum_{t=1}^{T}B(\Sigma_t,\Sigma_{t-1})\\
    &\qquad\qquad
        +\frac{\sqrt{d}M^4\widetilde{C_0}\sqrt{M^2+\epsilon}}{\epsilon^{\nicefrac{3}{2}}C_0(1-\beta)^2}\sum_{t=1}^{T}\eta_t(1-\theta_t)\\
    &\qquad\qquad
        +\left(
            \frac{M^2L\widetilde{C_0}^2}{\epsilon C_0^2(1-\beta)^2\sqrt{M^2+\epsilon}}
            +\frac{M^2L\sqrt{M^2+\epsilon}}{2\epsilon(1-\beta)}
            \frac{M\widetilde{M}\beta\widetilde{C_0}^2\sqrt{M^2+\epsilon}}{\sqrt{\epsilon}C_0^2(1-\beta)^2}
        \right)\sum_{t=1}^{T}\eta_t^2\\
    &\leq
        \frac{\sqrt{M^2+\epsilon}}{1-\beta}\big(f(\bmx_0)-f^{\star}\big)
        +\frac{\sqrt{M^2+\epsilon}}{1-\beta}\sum_{t=1}^{T}B(\Sigma_t,\Sigma_{t-1})\\
    &\qquad\qquad
        +\frac{\sqrt{d}M^4\widetilde{C_0}\sqrt{M^2+\epsilon}}{\epsilon^{\nicefrac{3}{2}}C_0(1-\beta)^2}\sum_{t=1}^{T}\eta_t(1-\theta_t)\\
    &\qquad\qquad
        +\frac{M^2LC_0^2(2+\widetilde{M})\sqrt{M^2+\epsilon}}{\epsilon C_0^2(1-\beta)^2}\sum_{t=1}^{T}\eta_t^2,
\end{align}
where the last inequality uses $0\leq\beta<1$ and $0<C_0\leq\widetilde{C_0}$.
With
\begin{align}
    C_1&=\frac{\sqrt{M^2+\epsilon}}{1-\beta}\left(f(\bmx_0)-f^{\star}-\sum_{t=1}^{T}B(\Sigma_t,\Sigma_{t-1})\right)\\
    C_2&=\frac{\sqrt{d}M^4\widetilde{C_0}\sqrt{M^2+\epsilon}}{\epsilon^{\nicefrac{3}{2}}C_0(1-\beta)^2}\\
    C_3&=\frac{M^2LC_0^2(2+\widetilde{M})\sqrt{M^2+\epsilon}}{\epsilon C_0^2(1-\beta)^2}
\end{align}
we have
\begin{align}
\label{eqn:thm4_1}
    \sum_{t=1}^{T}\eta_t\mathbb{E}\big[\|\nabla f_{\Sigma_t}(\bmx_{t-1})\|^2\big]
    &\leq
        C_1
        +C_2\sum_{t=1}^{T}\eta_t(1-\theta_t)
        +C_3'\sum_{t=1}^{T}\eta_t^2.
\end{align}
Since
\begin{equation}
    \min_{1\leq t\leq T}\mathbb{E}\big[\|\nabla f_{\Sigma_t}(\bmx_{t-1})\|^2\big]\sum_{t=1}^{T}
    \leq\sum_{t=1}^{T}\mathbb{E}\big[\|\nabla f_{\Sigma_t}(\bmx_{t-1})\|^2\big],
\end{equation}
we have the result.
\end{proof}

The next result provides almost sure convergence of \agsadam.
The proof primarily relies on \eqref{eqn:thm4_1} in the proof of the previous theorem.

\begin{thm}[Theorem C.2 from \cite{gssgd}]
\label{thm:adam_theorem9}
Let $f=\frac{1}{K}\sum_{k=1}^{K}f_{k}$ be $L$-smooth and $f^*$ denote the minimum of $f$.
Suppose that $\mathbb{E}\big[\|\nabla f_k\|^2\big]\leq\lambda$ for any $k\in[K]$.
Let $(\bmx_t)_{t\geq 1}$ be generated by \agsadam.
Suppose
\begin{equation}
    \sum_{t=1}^{\infty}\frac{\eta_t}{\sum_{i=1}^{t-1}\eta_i}=\infty,
    \quad\sum_{t=1}^{\infty}\eta_t^2<\infty,
    \quad\sum_{t=1}^{\infty}\eta_t(1-\theta_t)<\infty,
\end{equation}
and $\eta_t$ is decreasing.
Assume that for some $\widetilde{M}>0$
\begin{equation}
    \widetilde{B}(\Sigma_{t+1},\Sigma_t)\leq \widetilde{M}\eta_t
    \text{ and }
    \sum_{t=1}^{\infty}B(\Sigma_t,\Sigma_{t-1})<\infty.
\end{equation}
Then for any $T\geq 1$, we have
\begin{equation}
    \min_{1\leq t\leq T}\|\nabla f_{\Sigma_{t+1}}(\bmx_t)\|^2=o\left(\frac{1}{\sum_{t=1}^T\eta_t}\right)\;\text{a.s.}
\end{equation}
\end{thm}

\begin{proof}
The proof is exactly the same, just replace $f(\bmx^k)$ with $f_{\Sigma_{t+1}}(\bmx_t)$.
Also, in order to use Lebesgue's Monotone Convergence Theorem, we need to use our assumption about the summability of $B(\Sigma_{t+1},\Sigma_t)$ and use equation~\eqref{eqn:thm4_1}.
\end{proof}



We are finally ready to show that the iterates from \agsadam converge to a stationary point a.s. in $\bmx$.

\begin{proof}[Proof of Theorem~\ref{thm:agsadam}]
Since $f$ satisfies the assumptions of Theorem~\ref{thm:adam_theorem9}, from the proof of Theorem~\ref{thm:adam_theorem9} we have that
\begin{equation}
    \sum_{t=1}^{\infty}\eta_{t+1}\|\nabla f_{\Sigma_{t+2}}(\bmx_{t+1})\|^2<\infty\text{ a.s.}
\end{equation}
Since $f$ is $L$-smooth,
\begin{align}
    \Big|\|\nabla f_{\Sigma_{t+2}}(\bmx_{t+1})\|-\|\nabla f_{\Sigma_{t+1}}(\bmx_{t})\|\Big|
    &\leq\|\nabla f_{\Sigma_{t+2}}(\bmx_{t+1})-\nabla f_{\Sigma_{t+1}}(\bmx_{t})\|\\
    &\leq L\|\bmx_{t+1}-\bmx_t\|+\widetilde{B}(\Sigma_{t+1},\Sigma_t)\\
    &=L\eta_{t+1}\left\|\frac{\bmm_{t+1}}{\sqrt{\bmv_{t+1}+\epsilon}}\right\|+\widetilde{B}(\Sigma_{t+1},\Sigma_t)\\
    &\leq\frac{LM}{\sqrt{\epsilon}}\eta_{t+1}+\widetilde{B}(\Sigma_{t+1},\Sigma_t)\text{ a.s.}\\
    &\leq\frac{LM}{\sqrt{\epsilon}}\eta_{t+1}+c\eta_{t+1}\text{ a.s.}\\
    &=\eta_{t+1}\left(\frac{LM}{\sqrt{\epsilon}}+c\right).
\end{align}
By Lemma 21 of \cite{he2023convergence},
\begin{equation}
     \lim_{t\to\infty}\|\nabla f_{\Sigma_{t+1}}(\bmx_t)\|^2=0\text{ a.s.}.
\end{equation}
Since $\|\Sigma_t\|\to 0$, we know that
\begin{equation}
     \lim_{t\to\infty}\|\nabla f(\bmx_t)\|^2=0\text{ a.s.}.
\end{equation}
Since $\|\nabla f(\bmx_t)\|\leq M$ a.s., by Lebesgue's Dominated Convergence Theorem, we have that
\begin{equation}
    \lim_{t\to\infty}\mathbb{E}(\|\nabla f(\bmx_t)\|^2)
    =\mathbb{E}\left[\lim_{t\to\infty}\|\nabla f(\bmx_t)\|^2\right]
    =0.
\end{equation}
\end{proof}

\end{document}